\documentclass[12pt,a4paper]{amsart}
\usepackage[utf8]{inputenc}
\usepackage[T1]{fontenc}

\usepackage[cm]{fullpage}
\usepackage[dvipsnames]{xcolor}
\usepackage{url}

\usepackage{amsfonts, amsthm, amssymb,verbatim,amscd,amsmath, blindtext}

\usepackage{multirow,multicol}
\usepackage{graphicx}
\usepackage{charter,textcomp,csquotes,comment}
\usepackage{enumitem,yhmath}
\usepackage{ytableau}
\usepackage{float, pifont}

\usepackage{booktabs}
\usepackage{mathtools,nccmath}
\usepackage{float}
\usepackage[pagebackref,colorlinks=true,citecolor=blue,linkcolor=red,urlcolor=blue]{hyperref}
\newtheorem{theorem}{Theorem}[section]
\newtheorem{lemma}[theorem]{Lemma}
\newtheorem{corollary}[theorem]{Corollary}
\newtheorem{proposition}[theorem]{Proposition}

\theoremstyle{remark}
\newtheorem{remark}[theorem]{Remark}

\theoremstyle{definition}
\newtheorem{definition}[theorem]{Definition}

\newcommand{\R}{\mathbb{R}}
\newcommand{\Rs}{\mathcal{R}}
\newcommand{\1}{\mathbf{1}}

\newcommand{\bx}{\mathbf{x}}

\newcommand{\calR}{\mathcal{R}}
\newcommand{\tr}{\operatorname{tr}}

\newcommand{\norm}[1]{\left\|#1\right\|}

\title{Spectral Properties of the Resistance Laplacian with Applications to Data Clustering and Anomaly Detection}

\author[P. Deshpande]{Priyavrat Deshpande}
\address{Chennai Mathematical Institute, India}
\email{pdeshpande@cmi.ac.in}

\author[G. Lather]{Gargi Lather}
\address{Chennai Mathematical Institute, India}
\email{gargilather@gmail.com}
\dedicatory{The paper is dedicated to the memory of Professor R. B. Bapat}

\begin{document}

\begin{abstract}
The resistance Laplacian is a graph matrix associated with the effective resistance metric and provides a global counterpart of the classical graph Laplacian. 
Although it inherits several fundamental properties of the ordinary Laplacian, including a connected graph partitioning theorem analogous to that of Fiedler, its intrinsic spectral structure has remained largely unexplored.

In this paper, we develop a structural theory of the resistance Laplacian. 
We derive a canonical decomposition that separates its intrinsic, average, and deviation components, thereby revealing how the global geometry induced by effective resistance differs from the local geometry encoded by the ordinary Laplacian.
Building upon this decomposition, we establish several structural and spectral properties of the associated deviation operator, obtain variational characterizations of the largest eigenvalue and its corresponding eigenspace, and express the resistance Laplacian in Laplacian coordinates, thereby elucidating the relationship between the eigenspaces of the two operators. 

Finally, we formulate resistance-based graph partitioning objectives whose spectral relaxations recover the dominant eigenvector of the resistance Laplacian, providing a variational interpretation of the connected partition theorem. 
Experimental results on synthetic and real-life datasets demonstrate the effectiveness of the proposed framework for graph partitioning, data clustering, and exploratory anomaly detection.
\end{abstract}
\maketitle

\section{Introduction}
Spectral graph theory investigates the interplay between the combinatorial structure of a graph and the spectral properties of matrices naturally associated with it.
Among these, the graph Laplacian has emerged as one of the most fundamental objects, owing to its deep connections with graph connectivity, random walks, electrical networks, and discrete potential theory. 
Its eigenvalues and eigenvectors provide a geometric description of the underlying graph, revealing information about connectivity, diffusion, and community structure, and forming the basis of numerous spectral algorithms for graph partitioning, clustering, dimensionality reduction, and network analysis.
In particular, Fiedler's celebrated characterization of the second smallest Laplacian eigenvalue and its associated eigenvector established the foundation of modern spectral partitioning algorithms. We refer the reader to the monographs of Chung~\cite{Chung1997}, Mohar~\cite{Mohar1991}, and Brouwer and Haemers~\cite{brouwer2011spectra} for comprehensive accounts of spectral graph theory and its applications.

Many of these developments, however, are based on a fundamentally local notion of graph geometry. The graph Laplacian is constructed directly from adjacency information, and consequently its spectral properties primarily reflect local interactions between neighboring vertices. 
While this local viewpoint has proved remarkably successful, it does not explicitly capture the cumulative influence of multiple paths connecting distant parts of a graph. 
In many situations, particularly those involving information flow, network robustness, or global connectivity, one seeks a spectral description that reflects the graph as a whole rather than only its local adjacency structure.

One natural approach to incorporating global connectivity is through the effective resistance metric introduced by Klein and Randi\'c~\cite{Klein}. Viewing every edge as a unit electrical resistor, the effective resistance between two vertices measures the potential difference required to sustain a unit current between them. Unlike the classical shortest-path distance, effective resistance simultaneously incorporates the contribution of all connecting paths and therefore provides a genuinely global measure of accessibility within the graph. 
This perspective has led to a number of resistance-based graph invariants. 
See, for example, \cite{Bapat2010, bapat2014, parab2024}.

Effective resistance has proved to be a remarkably versatile graph concept, extending far beyond its origins in electrical network theory. 
The associated resistance matrix has been extensively investigated from combinatorial, algebraic, and spectral viewpoints, with connections to random walks, network reliability, chemical graph theory, and mathematical chemistry. 
Closely related is the Moore--Penrose inverse of the graph Laplacian, whose entries encode effective resistances and play a fundamental role in the study of resistance distances, Kirchhoff indices, spanning trees, diffusion processes, and discrete potential theory. These developments have established effective resistance and the Laplacian pseudoinverse as indispensable tools for understanding the global geometry of graphs.

Building on the classical Laplacian, Aouchiche and Hansen \cite{Aouchiche2013} introduced  the \emph{distance Laplacian matrix} as a spectral tool encoding metric information about the graph. 
For a connected graph $G$, the \emph{distance} $d_{ij}$ between 
vertices $i$ and $j$ is the length of a shortest path between them, and the  \emph{distance matrix} is $\mathcal{D} := [d_{ij}]$. 
The \emph{transmission} of vertex $i$ is  defined as $\mathrm{tr}_i := \sum_{j=1}^n d_{ij}$, and one sets $\mathrm{Tr} := \mathrm{diag}(\mathrm{tr}_1, \dots, \mathrm{tr}_n)$. 
The \emph{distance Laplacian matrix} of $G$ is then $L^D := \mathrm{Tr} - \mathcal{D}$.

Let $y = (y_1, \dots, y_n)^\top$ be an eigenvector corresponding to the largest 
eigenvalue of $L^D$, and define $Y_+ := \{i : y_i \geq 0\}$ and 
$Y_- := \{i : y_i \leq 0\}$. It has been proved for paths \cite{nath, bapat2014} that the subgraphs induced by
$Y_+$ and $Y_-$ are connected. Because resistance distance equals
graph distance on a tree, the result of \cite[Theorem~3.1]{gupta} also
yields the corresponding distance-Laplacian statement for trees.
However, this property does not hold for general graphs.

Following this paradigm, the resistance Laplacian, introduced independently by Parab and Dsouza \cite{parab2024}, and by Gupta, Lather, Balaji, and Kurata~\cite{gupta2} provides a spectral operator associated with this global geometry. 
It is obtained by replacing the shortest-path distance in the distance Laplacian construction with the effective resistance metric. Like the ordinary Laplacian, the resistance Laplacian is symmetric, positive semidefinite, and has the all-ones vector as an eigenvector corresponding to the eigenvalue zero. For general graphs, it can be viewed as a cycle‑sensitive deformation of the distance Laplacian.

More significantly, Gupta \emph{et al.} \cite[Theorem 3.1]{gupta} established that the sign partition induced by an eigenvector corresponding to its largest eigenvalue always produces connected induced subgraphs, providing an effective-resistance analogue of Fiedler's classical theorem.

Despite this striking partition theorem, the resistance Laplacian itself remains comparatively poorly understood. Unlike the ordinary Laplacian, whose spectral theory is now well developed, the resistance Laplacian is defined indirectly through effective resistance and therefore through the Moore--Penrose inverse of the graph Laplacian. Consequently, many fundamental questions remain open. What structural information is encoded by its spectrum? How does it relate to the ordinary Laplacian? Which graph parameters determine the deviation between their eigenspaces? More generally, what aspects of graph geometry are captured by the resistance Laplacian that are invisible to the ordinary Laplacian?

The present paper addresses these questions by developing a structural theory of the resistance Laplacian. Rather than viewing it solely as a matrix whose dominant eigenvector yields a graph partition, we study it as a spectral object in its own right. Our starting point is a canonical decomposition that separates the contribution of the Moore--Penrose inverse of the graph Laplacian from the contribution arising from the non-uniformity of the resistance centralities. This decomposition provides a transparent description of the geometry encoded by the resistance Laplacian and serves as the foundation for the spectral and variational results developed in the sequel.

The principal contribution of this paper is the development of a structural theory for the resistance Laplacian. We first derive a canonical decomposition expressing the resistance Laplacian as the sum of three naturally occurring components: an intrinsic term determined by the Moore--Penrose inverse of the graph Laplacian, an average correction depending only on the mean resistance centrality, and a deviation term capturing the non-uniformity of the resistance centralities. This decomposition provides a unified framework for understanding the relationship between the ordinary and resistance Laplacians and serves as the foundation for the subsequent analysis.

Building upon this decomposition, we develop the spectral theory of the deviation operator. We characterize its spectrum, determine its inertia, obtain explicit variational descriptions of its quadratic form, and identify the role played by the resistance deviations in governing its eigenstructure. Restricting the resistance Laplacian to the orthogonal complement of the all-ones vector then yields a particularly transparent representation in Laplacian coordinates, revealing how the resistance geometry modifies the classical Laplacian eigenspaces through a deviation-induced coupling term.

These structural results naturally lead to new algorithmic interpretations. 
We formulate resistance-based graph partitioning objectives whose spectral relaxations recover a dominant eigenvector of the resistance Laplacian, thereby providing a variational interpretation of the partitioning theorem of Gupta, Lather and Balaji \cite{gupta}. 
We further illustrate the resulting framework through graph partitioning, data clustering, and exploratory anomaly detection experiments, demonstrating that the resistance Laplacian provides a genuinely different geometric representation from that arising from the ordinary graph Laplacian.

The remainder of the paper is organized as follows. Section \ref{sec:spectral} develops the canonical decomposition of the resistance Laplacian together with its structural, spectral, and variational theory. 
Section \ref{sec:applications} applies this framework to graph partitioning, data clustering, and exploratory anomaly detection, and discusses the associated computational implications. 
We conclude in Section \ref{sec:conclusion} with directions for future research.

\section{Structural and Spectral Theory of the Resistance Laplacian}\label{sec:spectral}

Throughout this paper, unless stated otherwise, $G=(V,E)$ denotes a
finite, simple, connected, undirected, unweighted graph on $n$
vertices, where $V=\{1,\ldots,n\}$. Let $A=(a_{ij})_{i,j=1}^{n}$ be the adjacency matrix of $G$, where $a_{ij}=1$ if
$\{i,j\}\in E$ and $a_{ij}=0$ otherwise. Let $D_G=\operatorname{diag}\bigl(\deg(1),\ldots,\deg(n)\bigr)$ be the degree matrix of $G$. The \emph{combinatorial Laplacian matrix}
of $G$ is
\[
L=D_G-A.
\]

Since $G$ is connected, $\ker L=\operatorname{span}\{\mathbf{1}\},$ where $\mathbf{1}$ denotes the all-ones vector. The
\emph{Moore--Penrose inverse} of $L$, denoted by $L^\dagger$, is the
unique symmetric matrix satisfying
\[
LL^\dagger=L^\dagger L
=
I-\frac{1}{n}\mathbf{1}\mathbf{1}^{\mathsf T},
\qquad
L^\dagger\mathbf{1}=\mathbf{0},
\]
where $I$ denotes the identity matrix of order $n$.

Regarding every edge of $G$ as a unit resistor, the effective
resistance between vertices $i$ and $j$ is given by
\cite[Section~9.1]{Bapat2010}
\[
r_{ij}
=
(L^\dagger)_{ii}
+
(L^\dagger)_{jj}
-
2(L^\dagger)_{ij}.
\]
The corresponding \emph{resistance matrix} is $R=(r_{ij})_{1\leq i,j\leq n}.$ The \emph{resistance Laplacian} of $G$, introduced independently in
\cite{gupta2,parab2024}, is defined by
\begin{equation}
\mathcal{R}
=
\operatorname{diag}(R\mathbf{1})-R.
\label{eqn:resistance laplacian}
\end{equation}
The resistance Laplacian has previously been studied from the perspective of explicit matrix representations, spectral properties for special graph families, and associated graph energies \cite{parab2024}. 
In contrast, our objective is to understand its intrinsic structure for arbitrary connected graphs. 
The key observation underlying our approach is that the resistance Laplacian admits a canonical decomposition governed by the diagonal of the Moore--Penrose inverse of the graph Laplacian.



\subsection{Canonical Decomposition and the Deviation Operator}

The deviation term in the canonical decomposition is governed entirely by the diagonal of the Moore--Penrose inverse. 
We therefore begin by isolating this diagonal and decomposing it into its average and deviation components.
Let
\[
D=\operatorname{diag}(d_1,\ldots,d_n), \;\text{where} \; d_i=(L^\dagger)_{ii}.
\]
The vector $d=(d_1,\ldots,d_n)^\top$ is called the \emph{resistance centrality vector},
while $D$ is the \emph{resistance diagonal matrix}.
We further define the \emph{average resistance centrality} as
\[
\bar d :=\frac1n\operatorname{tr}(L^\dagger),
\]
and define the \emph{resistance deviation matrix} $\Delta$ as
\[
\Delta := D-\bar d I.
\]
Its diagonal entries $\delta_1, \dots, \delta_n$ are called \emph{resistance deviations}. 
Observe that $\Delta$ is trace-free. 
We write
\[
\delta=(\delta_1,\ldots,\delta_n)^{\top}.
\]
For each vertex $i$, let
\[
t_i=\sum_{j=1}^{n}r_{ij}
\]
denote its \emph{resistance transmission}, and let
\[
\bar t=\frac{1}{n}\sum_{i=1}^{n}t_i
\]
denote the average resistance transmission. Since
$L^\dagger\mathbf{1}=\mathbf{0}$, the effective-resistance formula gives
\[
t_i=nd_i+\operatorname{tr}(L^\dagger),
\qquad
\bar t=2\operatorname{tr}(L^\dagger),
\qquad
t_i-\bar t=n\delta_i.
\]
Thus, $\delta_i$ is the resistance-transmission deviation of vertex $i$,
normalized by the factor $n$. Vertices with $\delta_i<0$ are relatively
central in the resistance metric and will be referred to as \emph{globally
accessible}, whereas vertices with $\delta_i>0$ are relatively peripheral
and will be referred to as \emph{globally remote}.
Though these terms are introduced here, we will need them in the next section.

\begin{figure}
    \centering
    \includegraphics[width=0.5\linewidth]{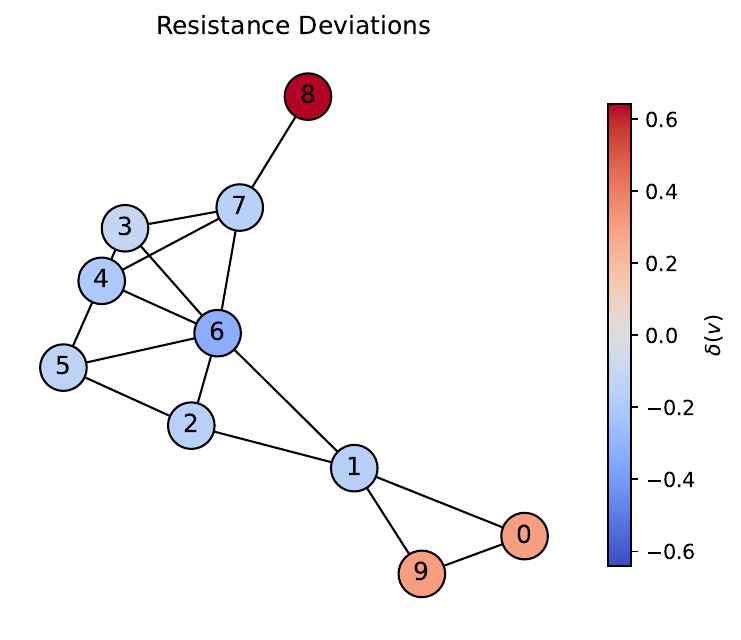}
    \caption{Resistance deviation on a graph}
    \label{fig:wheeldelta}
\end{figure}

Figure \ref{fig:wheeldelta} illustrates the resistance deviations on a sample graph. 
The hub-like vertex (numbered $6$) has the most negative deviation, whereas the leaf vertex (numbered $8$) has the largest positive deviation.

Finally, let
\[
P=I-\frac1nJ, 
\]
where $J=\mathbf{1}\mathbf{1}^{\top}$ is the all-ones matrix. Then $P$
is the orthogonal projection onto the subspace $\mathbf{1}^\perp$.

\begin{lemma}
The resistance matrix admits the decomposition

\[
R=DJ+JD-2L^\dagger.
\]
\end{lemma}

\begin{proof}
First, write the effective-resistance formula in matrix form as follows
\[
R=d\mathbf{1}^\top+\mathbf{1}d^\top-2L^\dagger.
\]
Using $d=D\mathbf{1}$, we obtain
\[
d\mathbf{1}^\top=D\mathbf{1}\mathbf{1}^\top=DJ,
\]
and similarly,
\[
\mathbf{1}d^\top=JD.
\]
Substituting these expressions gives the desired identity.
\end{proof}

The preceding lemma immediately yields a convenient matrix expression for
the resistance Laplacian.

\begin{proposition}
The resistance Laplacian satisfies

\[ \mathcal{R} = 2L^\dagger + nD + \operatorname{tr}(L^\dagger)I - DJ -JD.
\]
\end{proposition}

\begin{proof}
In view of Equation \ref{eqn:resistance laplacian}, we compute the row sums of $R$.
Using the previous lemma, we get:

\[
R\mathbf{1}
=
nD\mathbf{1}
+
Jd.
\]

Finally,

\[
Jd
=
(\mathbf{1}^\top d)\mathbf{1}
=
\operatorname{tr}(L^\dagger)\mathbf{1},
\]

and therefore

\[
\operatorname{diag}(R\mathbf{1})
=
nD+\operatorname{tr}(L^\dagger)I.
\]

Substituting into the definition of $\mathcal{R}$ completes the proof.
\end{proof}

The preceding expression can be simplified further by separating the
average and fluctuating parts of the resistance diagonal matrix.

Using the definition of the resistance deviation matrix, introduced above, we get
\[
DJ+JD = 2\bar d\,J+\Delta J+J\Delta,
\]
and
\[
nD+\operatorname{tr}(L^\dagger)I = 2n\bar d\,I+n\Delta.
\]

Substituting these identities into the previous proposition yields
\[
\mathcal{R} = 2L^\dagger + 2n\bar d\,I - 2\bar d\,J + n\Delta - \Delta J - J\Delta. 
\]

We have
\[
J=n(I-P),
\]
therefore
\[
\Delta J=n\Delta(I-P) \;\text{and}\; J\Delta=n(I-P)\Delta.
\]

Hence
\[
\Delta J+J\Delta
=
2n\Delta
-
n(\Delta P+P\Delta).
\]

Substituting this identity gives
\[
\mathcal{R} = 2L^\dagger + 2n\bar d\,P + n(\Delta P+P\Delta-\Delta).
\]

\begin{definition}
The matrix
\[
T=\Delta P+P\Delta-\Delta
\]
will be called the \emph{resistance deviation operator} associated
with $G$.
\end{definition}

Since $\Delta$ is diagonal and trace-free,
\[
J\Delta J
=
\mathbf{1}
\bigl(\mathbf{1}^{\top}\Delta\mathbf{1}\bigr)
\mathbf{1}^{\top}
=
\operatorname{tr}(\Delta)J
=
\mathbf{0}.
\]
Using $P=I-\frac{1}{n}J$, we therefore obtain
\[
\begin{aligned}
P\Delta P
&=
\left(I-\frac{1}{n}J\right)
\Delta
\left(I-\frac{1}{n}J\right)\\
&=
\Delta-\frac{1}{n}\Delta J-\frac{1}{n}J\Delta
+\frac{1}{n^2}J\Delta J\\
&=
\Delta-\frac{1}{n}\Delta J-\frac{1}{n}J\Delta\\
&=
\Delta P+P\Delta-\Delta\\
&=
T.
\end{aligned}
\]
Thus,
\[
T=P\Delta P.
\]

We say that $G$ is \emph{resistance-transmission regular} if
\[
t_1=t_2=\cdots=t_n.
\]
The resistance deviation operator vanishes precisely when $G$ is
resistance-transmission regular. More explicitly,
\[
T=\mathbf{0}
\quad\Longleftrightarrow\quad
\Delta=\mathbf{0}
\quad\Longleftrightarrow\quad
t_1=t_2=\cdots=t_n.
\]
Indeed, $\Delta=\mathbf{0}$ immediately implies $T=\mathbf{0}$.
Conversely, for $n\geq 3$, the diagonal entries of $T$ are
\[
T_{ii}=\left(1-\frac{2}{n}\right)\delta_i.
\]
Hence, $T=\mathbf{0}$ implies $\delta_i=0$ for every $i$, and therefore
$\Delta=\mathbf{0}$. The cases $n=1$ and $n=2$ are immediate, since
every connected graph of order at most two is
resistance-transmission regular. Finally, the transmission identity
established above shows that $\Delta=\mathbf{0}$ is equivalent to
$t_1=\cdots=t_n$.

We have therefore established the following canonical decomposition.

\begin{theorem}[Canonical decomposition]
\label{canonical decomposition}
Let $G$ be a connected graph. Then
\[
\mathcal{R}=2L^\dagger+2n\bar d\,P+nT.
\]
\end{theorem}

The first term in the above decomposition is intrinsic, the second is
isotropic on $\mathbf{1}^{\perp}$, while the third measures departures
from resistance-transmission regularity. Note that
\[
n\bar d=\tr(L^\dagger).
\]
The canonical decomposition isolates the contribution of the resistance
deviations in the operator $T$. The remainder of this section develops
the algebraic, spectral, and variational properties of $T$. We begin
with its spectrum. By the identity $T=P\Delta P$ established above,
the restriction of $T$ to $\mathbf{1}^{\perp}$ is the compression of
the diagonal matrix $\Delta$ to $\mathbf{1}^{\perp}$. Consequently,
its eigenvectors in $\mathbf{1}^{\perp}$ arise either from vectors
supported on a single level set of the deviations and having coordinate
sum zero, or from vectors whose coordinates are proportional to
$(\delta_i-\lambda)^{-1}$ and satisfy a global zero-sum condition.
The following theorem makes this description precise.

\begin{theorem}\label{thm:spectrumT}
Let $\alpha_1<\cdots<\alpha_k$ be the distinct resistance deviations,
where $\alpha_i$ occurs with multiplicity $m_i$, and let
\[
I_i=\{j\in\{1,\ldots,n\}:\delta_j=\alpha_i\}.
\]
Then $T\mathbf{1}=\mathbf{0}$. Moreover, for each $i$, the set
\[
E_i=
\left\{
x\in\mathbb{R}^n:
\operatorname{supp}(x)\subseteq I_i,\ 
\mathbf{1}^{\top}x=0
\right\}
\]
is precisely the set of solutions in $\mathbf{1}^{\perp}$ of
\[
Tx=\alpha_i x,
\]
and
\[
\dim E_i=m_i-1.
\]

The remaining $k-1$ eigenvalues of
$\left.T\right|_{\mathbf{1}^{\perp}}$ are simple and are precisely the
roots of
\[
f(\lambda)
=
\sum_{j=1}^{n}\frac{1}{\delta_j-\lambda}
=
\sum_{i=1}^{k}\frac{m_i}{\alpha_i-\lambda}.
\]
For each $i\in\{1,\ldots,k-1\}$, exactly one such root lies in
$(\alpha_i,\alpha_{i+1})$. Together with the eigenvalue $0$
corresponding to $\mathbf{1}$, these eigenvalues constitute the entire
spectrum of $T$.
\end{theorem}

\begin{proof}
By the identity $T=P\Delta P$, we have
$T\mathbf{1}=\mathbf{0}$. Since $T$ is symmetric,
$\mathbf{1}^{\perp}$ is invariant under $T$, and
\[
\mathbb{R}^n
=
\operatorname{span}\{\mathbf{1}\}
\oplus
\mathbf{1}^{\perp}
\]
is an orthogonal decomposition into $T$-invariant subspaces.

Let $x\in\mathbf{1}^{\perp}$ satisfy
\[
Tx=\lambda x.
\]
We obtain
\[
Tx
=
P\Delta x
=
\Delta x-\frac{1}{n}\mathbf{1}\delta^{\top}x.
\]
Set
\[
c=\frac{1}{n}\delta^{\top}x.
\]
The eigenvalue equation is then equivalent to
\[
(\delta_j-\lambda)x_j=c,
\qquad
j=1,\ldots,n.
\]

Suppose first that $c=0$. Then
\[
(\delta_j-\lambda)x_j=0
\]
for every $j$. Hence, if $x\neq 0$, then $\lambda=\alpha_i$ for some
$i$, and the support of $x$ is contained in $I_i$. Since
$x\in\mathbf{1}^{\perp}$, we have
$\mathbf{1}^{\top}x=0$, and therefore $x\in E_i$.

Conversely, if $x\in E_i$, then
\[
\delta^{\top}x
=
\alpha_i\mathbf{1}^{\top}x
=
0,
\]
and hence
\[
Tx
=
P\Delta x
=
\Delta x
=
\alpha_i x.
\]
Thus, $E_i$ is precisely the set of solutions in
$\mathbf{1}^{\perp}$ of $Tx=\alpha_i x$. Since the coordinates of a
vector in $E_i$ are supported on $I_i$ and have sum zero,
\[
\dim E_i=m_i-1.
\]

Suppose now that $c\neq 0$. Then
$\lambda\notin\{\alpha_1,\ldots,\alpha_k\}$, and
\[
x_j=\frac{c}{\delta_j-\lambda}.
\]
Since an eigenvector may be rescaled, we may assume that $c=1$.
The condition $x\in\mathbf{1}^{\perp}$ then becomes
\[
\sum_{j=1}^{n}\frac{1}{\delta_j-\lambda}=0.
\]

Conversely, suppose that
$\lambda\notin\{\alpha_1,\ldots,\alpha_k\}$ satisfies this equation,
and define
\[
x_j=\frac{1}{\delta_j-\lambda},
\qquad
j=1,\ldots,n.
\]
Then $\mathbf{1}^{\top}x=0$. Moreover,
\[
\frac{1}{n}\delta^{\top}x
=
\frac{1}{n}
\sum_{j=1}^{n}\frac{\delta_j}{\delta_j-\lambda}
=
\frac{1}{n}
\sum_{j=1}^{n}
\left(
1+\frac{\lambda}{\delta_j-\lambda}
\right)
=
1.
\]
Therefore,
\[
Tx
=
\Delta x-\mathbf{1}
=
\lambda x.
\]
Thus, the eigenvalues arising from the case $c\neq 0$ are precisely the
roots of $f$.

On each interval $(\alpha_i,\alpha_{i+1})$, the function $f$ is
continuously differentiable and satisfies
\[
f'(\lambda)
=
\sum_{j=1}^{n}\frac{1}{(\delta_j-\lambda)^2}
>0.
\]
Furthermore,
\[
\lim_{\lambda\to\alpha_i^+}f(\lambda)=-\infty,
\qquad
\lim_{\lambda\to\alpha_{i+1}^-}f(\lambda)=+\infty.
\]
Hence, $f$ has exactly one root in each interval
$(\alpha_i,\alpha_{i+1})$. Each root is simple because
$f'(\lambda)>0$.

For any such root $\lambda$, the relation
\[
x_j=\frac{c}{\delta_j-\lambda}
\]
shows that the corresponding eigenspace in $\mathbf{1}^{\perp}$ is
one-dimensional. Since $T$ is symmetric, $\lambda$ is a simple
eigenvalue of $\left.T\right|_{\mathbf{1}^{\perp}}$.

There are no roots in $(-\infty,\alpha_1)$, since $f(\lambda)>0$ on
that interval, and there are no roots in $(\alpha_k,\infty)$, since
$f(\lambda)<0$ there. Thus, $f$ has exactly $k-1$ roots.

Finally,
\[
\sum_{i=1}^{k}(m_i-1)+(k-1)=n-1.
\]
Therefore, the eigenvectors described above account for all of
$\mathbf{1}^{\perp}$. Together with the eigenvector $\mathbf{1}$,
they account for the entire spectrum of $T$.
\end{proof}

\begin{corollary}\label{funda subspaces T}
Let
\[
Z=\{i\in\{1,\ldots,n\}:\delta_i=0\},
\qquad
m_0=|Z|,
\]
and let $e_i$ denote the $i$th standard basis vector of
$\mathbb{R}^n$. When $m_0=0$, define
\[
s=\sum_{i=1}^{n}\frac{1}{\delta_i}.
\]

\begin{enumerate}
\item[(i)]
If $m_0>0$, then
\[
\ker T
=
\operatorname{span}\{\mathbf{1}\}
\oplus
\bigl(\ker\Delta\cap\mathbf{1}^{\perp}\bigr),
\]
\[
\operatorname{Im}T
=
\operatorname{span}\{Pe_i:i\notin Z\}
=
\left\{
y\in\mathbf{1}^{\perp}:
y_i=y_j\text{ for all }i,j\in Z
\right\},
\]
and
\[
\operatorname{rank}T=n-m_0.
\]

\item[(ii)]
If $m_0=0$ and $s\neq 0$, then
\[
\ker T=\operatorname{span}\{\mathbf{1}\},
\qquad
\operatorname{Im}T=\mathbf{1}^{\perp},
\qquad
\operatorname{rank}T=n-1.
\]

\item[(iii)]
If $m_0=0$ and $s=0$, then
\[
\ker T
=
\operatorname{span}
\left\{
\mathbf{1},\Delta^{-1}\mathbf{1}
\right\},
\]
\[
\operatorname{Im}T
=
\left\{
y\in\mathbb{R}^n:
\mathbf{1}^{\top}y=0,\ 
\mathbf{1}^{\top}\Delta^{-1}y=0
\right\},
\]
and
\[
\operatorname{rank}T=n-2.
\]
\end{enumerate}
\end{corollary}

\begin{proof}
Every $x\in\mathbb{R}^n$ can be written uniquely as
\[
x=a\mathbf{1}+y,
\qquad
y\in\mathbf{1}^{\perp}.
\]
Since $T\mathbf{1}=0$, we have $Tx=0$ if and only if $Ty=0$.
As in the proof of Theorem~\ref{thm:spectrumT}, the equation $Ty=0$
is equivalent to
\[
\delta_i y_i=c,
\qquad
i=1,\ldots,n,
\]
where
\[
c=\frac{1}{n}\delta^{\top}y.
\]

Suppose first that $m_0>0$. Choosing any $i\in Z$ in the equation
$\delta_i y_i=c$ gives $c=0$. It follows that $y_i=0$ whenever
$i\notin Z$. Hence
\[
y\in\ker\Delta\cap\mathbf{1}^{\perp}.
\]
Conversely, every vector in
$\ker\Delta\cap\mathbf{1}^{\perp}$ belongs to $\ker T$. Therefore,
\[
\ker T
=
\operatorname{span}\{\mathbf{1}\}
\oplus
\bigl(\ker\Delta\cap\mathbf{1}^{\perp}\bigr).
\]
Since
\[
\dim\bigl(\ker\Delta\cap\mathbf{1}^{\perp}\bigr)=m_0-1,
\]
we obtain
\[
\dim\ker T=m_0
\]
and therefore
\[
\operatorname{rank}T=n-m_0.
\]

Suppose next that $m_0=0$. Then $\Delta$ is invertible, and the
equations $\delta_i y_i=c$ give
\[
y=c\Delta^{-1}\mathbf{1}.
\]
Since $y\in\mathbf{1}^{\perp}$,
\[
0
=
\mathbf{1}^{\top}y
=
c\mathbf{1}^{\top}\Delta^{-1}\mathbf{1}
=
c\sum_{i=1}^{n}\frac{1}{\delta_i}
=
cs.
\]
If $s\neq 0$, then $c=0$ and hence $y=0$. Thus
\[
\ker T=\operatorname{span}\{\mathbf{1}\},
\]
which gives
\[
\operatorname{rank}T=n-1.
\]

If $s=0$, then
\[
\Delta^{-1}\mathbf{1}\in\mathbf{1}^{\perp},
\]
and every scalar multiple of $\Delta^{-1}\mathbf{1}$ belongs to
$\ker T$. Therefore,
\[
\ker T
=
\operatorname{span}
\left\{
\mathbf{1},\Delta^{-1}\mathbf{1}
\right\}.
\]
The two spanning vectors are linearly independent, so
\[
\dim\ker T=2
\]
and
\[
\operatorname{rank}T=n-2.
\]

Finally, since $T$ is symmetric,
\[
\operatorname{Im}T=(\ker T)^{\perp}.
\]
The image formulas in parts \textup{(ii)} and \textup{(iii)} follow
immediately by taking orthogonal complements.

For part \textup{(i)}, a vector orthogonal to
$\ker\Delta\cap\mathbf{1}^{\perp}$ must have equal coordinates on
$Z$. Hence,
\[
(\ker T)^{\perp}
=
\left\{
y\in\mathbf{1}^{\perp}:
y_i=y_j\text{ for all }i,j\in Z
\right\}.
\]
Each vector $Pe_i$ with $i\notin Z$ belongs to this space. These
vectors are linearly independent because $Z\neq\varnothing$, and
their number is $n-m_0$, which equals $\operatorname{rank}T$.
Therefore,
\[
\operatorname{Im}T
=
\operatorname{span}\{Pe_i:i\notin Z\}.
\]
This completes the proof.
\end{proof}

Recall that the inertia of a real symmetric matrix $A$ is the triple
\[
\operatorname{In}(A)=(e_+,e_-,e_0),
\]
where $e_+$ and $e_-$ are the numbers of positive and negative
eigenvalues of $A$, respectively, and $e_0$ is the multiplicity of
the eigenvalue $0$.
The following result states that the inertia of $T$ is completely determined by that of $\Delta$ and the actual resistance deviations. 

\begin{lemma}\label{sign lemma}
Assume that $0$ does not occur as a resistance deviation. Let $\lambda_*$
denote the unique eigenvalue of $\left.T\right|_{\mathbf{1}^{\perp}}$
lying between the largest negative and the smallest positive resistance
deviations. Then
\[
\operatorname{sgn}(\lambda_*)
=
-\operatorname{sgn}\bigl(f(0)\bigr)
=
-\operatorname{sgn}\left(\sum_{i=1}^{n}\frac{1}{\delta_i}\right).
\]
\end{lemma}

\begin{proof}
Let
\[
\alpha_1<\alpha_2<\cdots<\alpha_k
\]
be the distinct resistance deviations, and choose $r$ such that
$\alpha_r<0<\alpha_{r+1}$. By Theorem~\ref{thm:spectrumT},
$\lambda_*$ is the unique root of $f$ in
$(\alpha_r,\alpha_{r+1})$, and $f$ is strictly increasing on this
interval. Since $0$ lies in the same interval, $\lambda_*>0$ when
$f(0)<0$, $\lambda_*<0$ when $f(0)>0$, and $\lambda_*=0$ when
$f(0)=0$. Finally,
\[
f(0)=\sum_{i=1}^{n}\frac{1}{\delta_i},
\]
which proves the result.
\end{proof}

\begin{theorem}\label{inertia of T}
Let $n_+$ denote the number of positive deviations and $n_-$ the number
of negative deviations. When $0$ does not occur as a resistance
deviation, there are the following three possibilities for the inertia
of $T$:
\begin{enumerate}
\item If $f(0)<0$, then
$\operatorname{In}(T)=(n_+,n_--1,1)$.
\item If $f(0)>0$, then
$\operatorname{In}(T)=(n_+-1,n_-,1)$.
\item If $f(0)=0$, then
$\operatorname{In}(T)=(n_+-1,n_--1,2)$.
\end{enumerate}
If $0$ occurs as a resistance deviation with multiplicity $n_0$, then
\[
\operatorname{In}(T)=(n_+,n_-,n_0).
\]
\end{theorem}

\begin{proof}
We use the notation of Theorem~\ref{thm:spectrumT}. Each deviation
$\alpha_i$ of multiplicity $m_i$ contributes an eigenspace of dimension
$m_i-1$, and each interval $(\alpha_i,\alpha_{i+1})$ contains exactly
one simple eigenvalue of $\left.T\right|_{\mathbf{1}^{\perp}}$.

Suppose first that $0$ is not a resistance deviation. The level-set
eigenvalues and the roots in the intervals contained in
$(-\infty,0)$ account for $n_--1$ negative eigenvalues. Similarly,
the level-set eigenvalues and the roots in the intervals contained in
$(0,\infty)$ account for $n_+-1$ positive eigenvalues. The remaining
root is $\lambda_*$. Lemma~\ref{sign lemma} determines its sign, and
the three stated inertia formulas follow after including the eigenvalue
$0$ corresponding to $\mathbf{1}$.

Now suppose that $0$ is a resistance deviation with multiplicity $n_0$.
If $\Delta=\mathbf{0}$, then $T=\mathbf{0}$ and the result is immediate.
Otherwise, since $\Delta$ is trace-free, there are both positive and
negative deviations. Write the distinct resistance deviations as
\[
\alpha_1<\cdots<\alpha_{r-1}<0=\alpha_r
<\alpha_{r+1}<\cdots<\alpha_k.
\]
The zero deviation contributes an eigenspace of dimension $n_0-1$;
together with the eigenvalue corresponding to $\mathbf{1}$, this gives
nullity $n_0$. The negative level-set eigenvalues and the roots in the
intervals to the left of $0$ account for $n_-$ negative eigenvalues.
Likewise, the positive level-set eigenvalues and the roots in the
intervals to the right of $0$ account for $n_+$ positive eigenvalues.
Hence,
\[
\operatorname{In}(T)=(n_+,n_-,n_0).
\]
\end{proof}

\begin{corollary}\label{indefiniteness}
The resistance deviation operator $T$ is indefinite unless
$\Delta=\mathbf{0}$. Equivalently, $T$ is semidefinite if and only if
$\Delta=\mathbf{0}$.
\end{corollary}

\begin{proof}
If $\Delta=\mathbf{0}$, then $T=P\Delta P=\mathbf{0}$, so $T$ is
semidefinite. Conversely, suppose that $\Delta\neq\mathbf{0}$. Since
$\Delta$ is trace-free, it has both positive and negative diagonal
entries. If $0$ occurs as a resistance deviation, the preceding theorem
gives
\[
\operatorname{In}(T)=(n_+,n_-,n_0),
\]
where $n_+>0$ and $n_->0$; hence $T$ is indefinite.

Now suppose that $0$ is not a resistance deviation. The preceding
theorem gives one of the three stated inertia formulas. Moreover,
\[
\operatorname{tr}(T)
=
\operatorname{tr}(P\Delta P)
=
\operatorname{tr}(\Delta P)
=0,
\]
and the equivalence established above gives $T\neq\mathbf{0}$. Therefore,
the positive and negative indices of $T$ must both be nonzero; otherwise,
$T$ would be a nonzero semidefinite matrix with trace zero. Thus, $T$ is
indefinite.
\end{proof}

The quadratic form associated with the deviation operator admits a particularly simple description.

\begin{proposition}
For every $x\in\mathbb{R}^n$,
\[
x^{\top}Tx
=
\sum_{i=1}^{n}\delta_i x_i^2
-
\frac{2}{n}
\left(\sum_{i=1}^{n}\delta_i x_i\right)
\left(\sum_{i=1}^{n}x_i\right).
\]
In particular, if $x\perp\mathbf{1}$, then
\[
x^{\top}Tx
=
\sum_{i=1}^{n}\delta_i x_i^2.
\]
\end{proposition}

\begin{proof}
Using
\[
T=\Delta-\frac1n(\delta\mathbf{1}^{\top}+\mathbf{1}\delta^{\top}),
\]
we compute
\[
x^{\top}Tx
=
x^{\top}\Delta x
-
\frac1n
x^{\top}
(\delta\mathbf{1}^{\top}+\mathbf{1}\delta^{\top})
x.
\]
Since
\[
x^{\top}\Delta x=\sum_{i=1}^{n}\delta_i x_i^2,
\]
and
\[
x^{\top}
(\delta\mathbf{1}^{\top}+\mathbf{1}\delta^{\top})
x
=
2(\delta^{\top}x)(\mathbf{1}^{\top}x),
\]
the desired identity follows immediately.
The second statement follows from the fact that
$\mathbf{1}^{\top}x=0$.
\end{proof}

\subsection{The Resistance Laplacian in Laplacian Coordinates}

Since
\[
L^\dagger\mathbf{1}=\mathbf{0}
\qquad\text{and}\qquad
\Rs\mathbf{1}=\mathbf{0},
\]
the subspace $\mathbf{1}^{\perp}$ is invariant under both
$L^\dagger$ and $\Rs$. Hence, the nonzero eigenvalues of $\Rs$ and
their corresponding eigenvectors can be studied by restricting $\Rs$
to $\mathbf{1}^{\perp}$. Combining the canonical decomposition from
Theorem~\ref{canonical decomposition} with the identity
$T=P\Delta P$ established above gives the following result.

\begin{corollary}\label{thm:restriction}
The restriction of the resistance Laplacian to
$\mathbf{1}^{\perp}$ is
\[
\left.\Rs\right|_{\mathbf{1}^{\perp}}
=
2n\bar d\,I_{\mathbf{1}^{\perp}}
+
2\left.L^\dagger\right|_{\mathbf{1}^{\perp}}
+
n\left.P\Delta P\right|_{\mathbf{1}^{\perp}},
\]
where $I_{\mathbf{1}^{\perp}}$ denotes the identity operator on
$\mathbf{1}^{\perp}$.
\end{corollary}

\begin{proof}
Restrict the canonical decomposition
\[
\Rs=2L^\dagger+2n\bar d\,P+nT
\]
to $\mathbf{1}^{\perp}$. Since
$\left.P\right|_{\mathbf{1}^{\perp}}=I_{\mathbf{1}^{\perp}}$ and
$T=P\Delta P$, the stated formula follows.
\end{proof}

The restriction separates the resistance Laplacian into an intrinsic component $2L^\dagger$, a scalar shift $2n\bar d I$, and a deviation component $nP\Delta P$. 
Since the scalar term affects only the eigenvalues, the eigenspaces are determined entirely by $2L^\dagger +nP\Delta P$. 
    
To better understand the interaction between the Laplacian geometry and the resistance
deviations, it is natural to express the restricted resistance Laplacian in the eigenbasis of
$L$. Let
\[
0=\lambda_1<\lambda_2\le\cdots\le\lambda_n
\]
be the Laplacian eigenvalues with corresponding orthonormal eigenvectors
$u_1,u_2,\ldots,u_n$, where
\[
u_1=\frac1{\sqrt n}\mathbf1.
\]
Since the restriction is taken to $\mathbf1^\perp$, the vectors
$u_2,\ldots,u_n$
form an orthonormal basis of this subspace.

Define
\[
U=\begin{bmatrix}
u_2&u_3&\cdots&u_n
\end{bmatrix},
\]
and introduce the matrices
\[
D_L=\operatorname{diag}\!\left(
\frac{1}{\lambda_2},
\ldots,
\frac{1}{\lambda_n}
\right),
\qquad
M_L=U^{\mathsf T}\Delta U.
\]

The diagonal matrix $D_L$ is the matrix representation of
$\left.L^\dagger\right|_{\mathbf{1}^{\perp}}$ in the Laplacian
eigenbasis, while $M_L$ is the matrix representation of
$\left.T\right|_{\mathbf{1}^{\perp}}$ in the same basis. Indeed, since
$PU=U$ and $U^{\mathsf T}P=U^{\mathsf T}$,
\[
U^{\mathsf T}TU
=
U^{\mathsf T}P\Delta PU
=
U^{\mathsf T}\Delta U
=
M_L.
\]

The next result gives the complete matrix representation of the resistance Laplacian on $\mathbf1^\perp$.

\begin{theorem}\label{thm:lapbasis}
With respect to the Laplacian eigenbasis
$\{u_2,\ldots,u_n\}$, the restriction of the resistance Laplacian is
represented by
\[
U^{\mathsf T}\Rs U
=
2n\bar d\,I_{n-1}+2D_L+nM_L.
\]
\end{theorem}

\begin{proof}
Since
\[
L^\dagger U=UD_L,
\]
we have
\[
U^{\mathsf T}L^\dagger U=D_L.
\]
Moreover, as established above,
\[
U^{\mathsf T}(P\Delta P)U=M_L.
\]
Applying $U^{\mathsf T}(\,\cdot\,)U$ to the decomposition in
Corollary~\ref{thm:restriction} gives
\[
U^{\mathsf T}\Rs U
=
2n\bar d\,U^{\mathsf T}U
+
2U^{\mathsf T}L^\dagger U
+
nU^{\mathsf T}(P\Delta P)U.
\]
Since $U^{\mathsf T}U=I_{n-1}$, the required expression follows.
\end{proof}
The preceding theorem separates the resistance Laplacian into two
fundamentally different components. The matrix $D_L$ is diagonal and
depends only on the spectrum of the ordinary graph Laplacian.
Consequently, it acts as a scalar on each Laplacian eigenspace. In
contrast, $M_L$ records how the resistance deviations interact with
the Laplacian eigenvectors and is generally non-diagonal. Thus, all
mixing between distinct Laplacian eigenspaces is encoded by the
cross-eigenspace entries of $M_L$.

This decomposition naturally raises the following question: to what
extent do the Laplacian eigenspaces remain invariant under the
resistance Laplacian? This question is governed entirely by whether
$M_L$ has nonzero entries between distinct Laplacian eigenspaces.

\begin{proposition}\label{prop:commutator}
Let
\[
C=[D_L,M_L]=D_LM_L-M_LD_L.
\]
Indexing the rows and columns of $C$ by the Laplacian eigenvectors
$u_2,\ldots,u_n$, we have
\[
C_{ij}
=
\left(
\frac{1}{\lambda_i}
-
\frac{1}{\lambda_j}
\right)
u_i^\top\Delta u_j,
\qquad
2\leq i,j\leq n.
\]
Consequently, $D_L$ and $M_L$ commute if and only if
\[
u_i^\top\Delta u_j=0
\]
whenever $\lambda_i\neq\lambda_j$.
\end{proposition}

\begin{proof}
Since $D_L$ is diagonal,
\[
(D_LM_L)_{ij}
=
\frac{1}{\lambda_i}(M_L)_{ij},
\]
whereas
\[
(M_LD_L)_{ij}
=
\frac{1}{\lambda_j}(M_L)_{ij}.
\]
Subtracting gives
\[
C_{ij}
=
\left(
\frac{1}{\lambda_i}
-
\frac{1}{\lambda_j}
\right)
(M_L)_{ij}.
\]
Finally,
\[
(M_L)_{ij}
=
u_i^\top\Delta u_j,
\]
which proves the stated formula.

The second assertion follows because
\[
\frac{1}{\lambda_i}-\frac{1}{\lambda_j}=0
\]
if and only if $\lambda_i=\lambda_j$.
\end{proof}

The proposition shows that the obstruction to the preservation of the
Laplacian eigenspace decomposition is precisely the presence of
nonzero matrix elements
\[
u_i^\top\Delta u_j
\]
between eigenspaces corresponding to distinct Laplacian eigenvalues.
If all such cross-eigenspace terms vanish, then the restrictions of
the resistance Laplacian and the ordinary graph Laplacian to
$\mathbf{1}^{\perp}$ commute and are simultaneously orthogonally
diagonalizable. Equivalently, every Laplacian eigenspace is invariant
under the resistance Laplacian. When a Laplacian eigenvalue has
multiplicity greater than one, the resistance Laplacian may further
decompose the corresponding Laplacian eigenspace.

The preceding matrix representation provides an algebraic description
of the resistance Laplacian. We now turn to its variational
interpretation. Since the resistance Laplacian is real symmetric, its
largest eigenvalue admits the usual Rayleigh characterization.
Expressing this optimization problem in terms of the canonical
decomposition provides a variational interpretation of the dominant
resistance-Laplacian eigenspace.

\begin{theorem}\label{thm:variational}
Let $\rho=\lambda_{\max}(\Rs).$ Then
\[
\rho
=
\max_{\substack{w\in\mathbf{1}^{\perp}\\ \|w\|=1}}
w^\top\Rs w.
\]
Moreover,
\[
\rho
=
2n\bar d
+
\max_{\substack{w\in\mathbf{1}^{\perp}\\ \|w\|=1}}
\left(
2w^\top L^\dagger w
+
nw^\top\Delta w
\right).
\]
\end{theorem}

\begin{proof}
Since $\Rs$ is positive semidefinite and
$\Rs\mathbf{1}=0$, its largest eigenvalue is attained on
$\mathbf{1}^{\perp}$. Therefore, the Courant--Fischer
characterization \cite[Theorem 4.2.6]{horn2012matrix} gives
\[
\rho
=
\max_{\substack{w\in\mathbf{1}^{\perp}\\ \|w\|=1}}
w^\top\Rs w.
\]

Substituting the decomposition from
Corollary~\ref{thm:restriction}, we obtain
\[
w^\top\Rs w
=
2n\bar d\,w^\top w
+
2w^\top L^\dagger w
+
nw^\top P\Delta Pw.
\]
Since $w\in\mathbf{1}^{\perp}$, we have $Pw=w$. Moreover, since
$P$ is symmetric,
\[
w^\top P\Delta Pw=w^\top\Delta w.
\]
Finally, $\|w\|=1$ implies $w^\top w=1$, and hence
\[
w^\top\Rs w
=
2n\bar d
+
2w^\top L^\dagger w
+
nw^\top\Delta w.
\]
Taking the maximum over all unit vectors in
$\mathbf{1}^{\perp}$ proves the result.
\end{proof}

\begin{figure}
    \centering
    \includegraphics[width=0.7\linewidth]{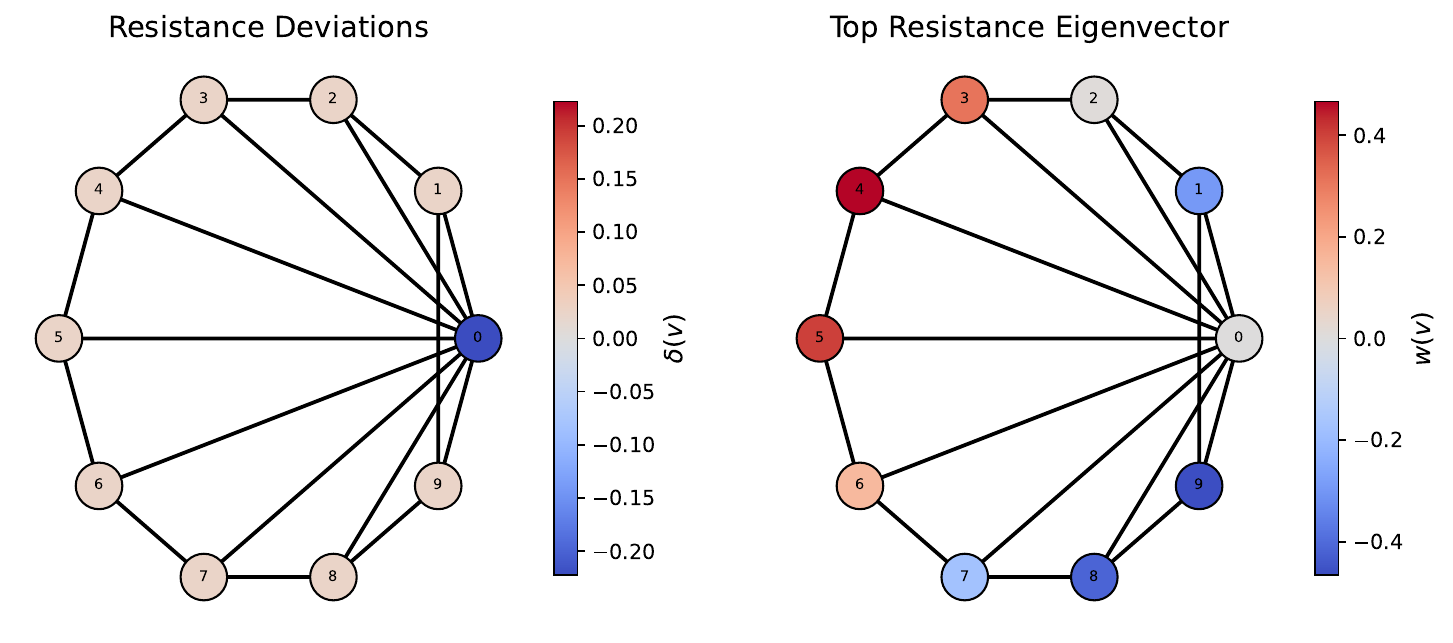}
    \caption{For a wheel graph, the left panel shows the resistance
    deviations, while the right panel shows a unit eigenvector
    associated with the largest resistance-Laplacian eigenvalue.}
    \label{fig:wheelcompare}
\end{figure}

For the eigenvector displayed in Figure~\ref{fig:wheelcompare}, the
globally remote vertices have the largest coordinate magnitudes,
whereas the globally accessible hub has a comparatively small
coordinate magnitude.

The variational characterization above naturally admits a coordinate
description in the Laplacian eigenbasis. Let
\[
w=\sum_{i=2}^{n}a_i u_i,
\]
where
\[
a=(a_2,\ldots,a_n)^\top\in\mathbb{R}^{n-1}.
\]
Since $\{u_2,\ldots,u_n\}$ is an orthonormal basis of
$\mathbf{1}^{\perp}$, the condition $\|w\|=1$ is equivalent to
$\|a\|=1$.

\begin{theorem}\label{thm:eigenbasisoptimization}
A vector $w=Ua$ is a unit eigenvector of $\Rs$ associated with
$\rho=\lambda_{\max}(\Rs)$ if and only if $a$ is a maximizer of
\[
\max_{\|a\|=1}
\left(
2a^\top D_La
+
na^\top M_La
\right).
\]
Equivalently, every such maximizer satisfies
\[
(2D_L+nM_L)a
=
(\rho-2n\bar d)a.
\]
\end{theorem}

\begin{proof}
Substituting
\[
w=Ua
\]
into the variational characterization established in
Theorem~\ref{thm:variational}, we obtain
\[
w^\top L^\dagger w
=
a^\top U^\top L^\dagger Ua
=
a^\top D_La,
\]
and similarly,
\[
w^\top\Delta w
=
a^\top U^\top\Delta Ua
=
a^\top M_La.
\]
Therefore,
\[
w^\top\Rs w
=
2n\bar d
+
2a^\top D_La
+
na^\top M_La.
\]
Since the first term is independent of $a$, maximizing
$w^\top\Rs w$ over unit vectors $w\in\mathbf{1}^{\perp}$ is
equivalent to maximizing
\[
2a^\top D_La+na^\top M_La
\]
subject to $\|a\|=1$.

The Lagrange multiplier condition for this constrained optimization
problem is
\[
(2D_L+nM_L)a=\mu a
\]
for some scalar $\mu$. Multiplying this equation by $a^\top$ and using
$\|a\|=1$ shows that
\[
\mu
=
2a^\top D_La
+
na^\top M_La.
\]
At a maximizer, $\mu$ is the largest eigenvalue of
$2D_L+nM_L$. Since
\[
\rho=2n\bar d+\mu,
\]
we obtain
\[
(2D_L+nM_L)a
=
(\rho-2n\bar d)a.
\]
The equivalence follows from the Rayleigh--Ritz characterization.
\end{proof}

The preceding theorem shows that the dominant resistance-Laplacian
eigenspace is determined by two distinct spectral mechanisms. The
diagonal matrix $D_L$ weights each Laplacian eigenmode independently,
assigning larger weights to modes corresponding to smaller positive
Laplacian eigenvalues. In contrast, $M_L$ may couple different
Laplacian modes according to the resistance deviations of the graph.
Thus, while $D_L$ preserves the ordinary Laplacian eigenspace
decomposition, $M_L$ may mix these modes in determining the dominant
eigenspace of the resistance Laplacian.

The coordinate description obtained above admits a simple
interpretation in vertex space. Although $M_L$ was introduced as the
matrix representation of
$\left.T\right|_{\mathbf{1}^{\perp}}$ in the Laplacian eigenbasis, its
quadratic form depends only on the graph signal itself. Consequently,
the contribution of the deviation term to the optimization problem
admits an intrinsic vertex-space description.

\begin{theorem}\label{thm:vertexenergy}
Let
\[
w=\sum_{i=2}^{n}a_i u_i
\]
be a unit vector in $\mathbf{1}^{\perp}$, and let
\[
a=(a_2,\ldots,a_n)^\top
\]
denote its coordinate vector in the Laplacian eigenbasis. Then
\[
a^\top M_La
=
\sum_{v\in V(G)}
\delta(v)\,w(v)^2.
\]
Consequently, the optimization problem in
Theorem~\ref{thm:eigenbasisoptimization} may be written as
\[
\max_{\substack{w\in\mathbf{1}^{\perp}\\\|w\|=1}}
\left(
2w^\top L^\dagger w
+
n\sum_{v\in V(G)}
\delta(v)w(v)^2
\right).
\]
\end{theorem}

\begin{proof}
Since
\[
w=Ua,
\]
we have
\[
w(v)
=
\sum_{i=2}^{n}a_i u_i(v).
\]
Indexing the rows and columns of $M_L$ by the Laplacian eigenvectors
$u_2,\ldots,u_n$, we have
\[
(M_L)_{ij}
=
u_i^\top\Delta u_j
=
\sum_{v\in V(G)}
\delta(v)\,u_i(v)u_j(v).
\]
Therefore,
\begin{align*}
a^\top M_La
&=
\sum_{i,j=2}^{n}
a_i(M_L)_{ij}a_j\\
&=
\sum_{v\in V(G)}
\delta(v)
\sum_{i,j=2}^{n}
a_i a_j u_i(v)u_j(v)\\
&=
\sum_{v\in V(G)}
\delta(v)
\left(
\sum_{i=2}^{n}
a_i u_i(v)
\right)^2\\
&=
\sum_{v\in V(G)}
\delta(v)\,w(v)^2.
\end{align*}
This proves the first assertion.

The second assertion follows by substituting this identity into the
optimization problem established in
Theorem~\ref{thm:eigenbasisoptimization}.
\end{proof}

The preceding theorem provides an intrinsic interpretation of the
resistance-deviation term. Rather than acting only as an abstract
perturbation in Laplacian coordinates, it measures how the squared
coordinates of the graph signal are distributed over the vertices.
Vertices with positive resistance deviations contribute positively to
the objective, whereas vertices with negative resistance deviations
contribute negatively. Thus, the deviation term favors graph signals
whose energy is concentrated on globally remote vertices.

This gives a conceptual description of the dominant
resistance-Laplacian eigenspace. Indeed,
\[
w^\top L^\dagger w
=
\sum_{i=2}^{n}\frac{a_i^2}{\lambda_i},
\]
so the first term favors low-frequency Laplacian modes, since modes
corresponding to smaller positive Laplacian eigenvalues receive larger
weights. In contrast, the deviation term favors the placement of
signal energy on globally remote vertices. The dominant
resistance-Laplacian eigenspace therefore results from a balance
between low-frequency Laplacian structure and the resistance-deviation
profile of the graph.

Figure~\ref{fig:ERtop} illustrates this behavior for the displayed
$10$-vertex graph. For the unit eigenvector shown, the most globally
remote vertex, vertex~$8$, has the largest coordinate magnitude,
whereas the globally accessible vertex~$6$ has a comparatively small
coordinate magnitude.

\begin{figure}
    \centering
    \includegraphics[width=0.5\linewidth]{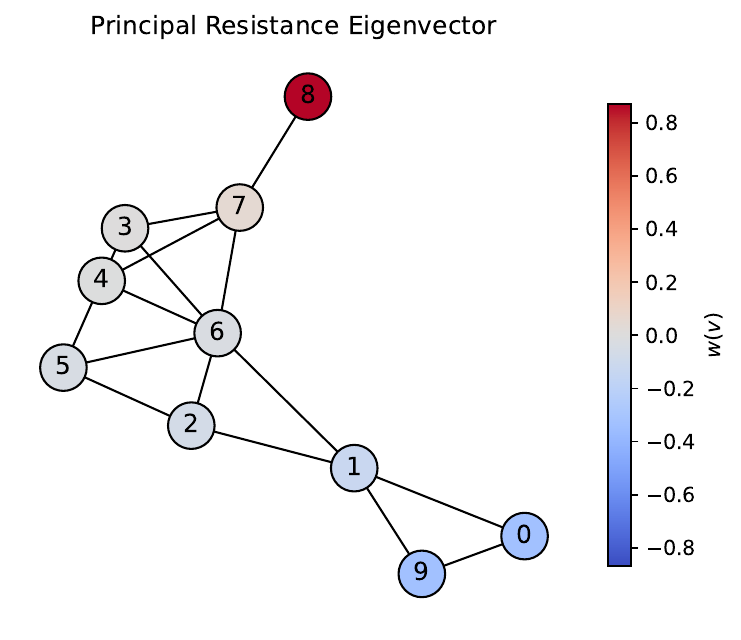}
    \caption{A unit eigenvector associated with the largest
    resistance-Laplacian eigenvalue of the displayed $10$-vertex
    graph.}
    \label{fig:ERtop}
\end{figure}


\section{Applications and Algorithmic Implications}\label{sec:applications}

The foundations of modern spectral graph partitioning were laid by Fiedler~\cite{Fiedler1973,Fiedler1975}, who established that the second smallest eigenvalue of the graph Laplacian encodes the connectivity of a graph and that the signs of the corresponding eigenvector induce connected graph partitions. 
These results provided the mathematical basis for a vast body of work on spectral partitioning, graph embeddings, and spectral clustering.

These ideas have evolved into spectral clustering, where graph Laplacian eigenvectors yield low-dimensional representations for standard clustering algorithms. 
This perspective is widely used in machine learning, computer vision, network science, and data analysis, inspiring extensive work on graph cuts, spectral relaxations, and graph embeddings~\cite{ShiMalik2000,NgJordanWeiss2001,vonLuxburg2007}, all based on the geometry of the classical graph Laplacian.

Motivated by the global geometry induced by effective resistance, Gupta, Lather and Balaji~\cite{gupta} proved a connected partition theorem for the resistance Laplacian: the sign partition of an eigenvector for its largest eigenvalue yields connected induced subgraphs, placing the dominant resistance eigenvector in a role analogous to the Fiedler vector for the ordinary Laplacian.
The aim of this section is to explore applications to data clustering and anomaly detection. 
At this point we do not claim to have new algorithms that beat the traditional methods. 
This is more of proof-of-concept exploration. 
The Python code is available at the following repository 

\url{https://github.com/priyavratd/resistance_laplacian}.

Classical spectral partitioning is closely related to RatioCut, Normalized Cut, and spectral clustering. These methods seek relaxations of combinatorial cut objectives whose optimizers are given by Laplacian eigenvectors. Our construction follows the same philosophy, but replaces edge-based separation by effective-resistance-based separation.

\subsection{A resistance-based graph partitioning objective}

Classical spectral partitioning solves a relaxed graph cut optimization problem, where the Fiedler vector is the optimizer. This vector yields an efficient approximation to otherwise NP-hard partitioning problems. This variational view is central to spectral graph theory and underlies many modern graph partitioning and clustering algorithms.

The variational characterization in Theorem \ref{thm:vertexenergy} motivates an analogous optimization principle for the resistance Laplacian. 
Instead of minimizing edge cuts defined by local adjacency, the objective uses vertex resistance deviations, capturing the graph’s global accessibility structure. 
This yields a resistance-based graph partitioning objective whose spectral relaxation is determined by the dominant eigenvector of the resistance Laplacian.

    
	\begin{definition}[Resistance cut]
Let $G=(V,E)$ be a connected graph, and let $r_{ij}$ denote the
effective resistance between vertices $i$ and $j$. For two nonempty
disjoint subsets $A,B\subseteq V$, the \emph{resistance cut} between
$A$ and $B$ is defined by
\[
\operatorname{ResCut}(A,B)
:=
\sum_{{i\in A,\,j\in B}} r_{ij}.
\]
When $V=A\mathbin{\dot\cup}B$, we call
$\operatorname{ResCut}(A,B)$ the resistance cut associated with the
bipartition $(A,B)$.
\end{definition}
	\begin{remark}
The classical graph cut measures the total weight of edges crossing the partition and therefore reflects only local adjacency. 
In contrast, the effective resistance $r_{ij}$ measures the global separation between two vertices by taking into account all paths joining them.
Consequently, $\mathrm{ResCut}(A, B)$ quantifies the total global separation between the two parts of the partition rather than merely the number or weight of crossing edges.
  \end{remark}

As in spectral clustering, we introduce cardinality-balanced cuts. 
    
	\begin{definition}[Resistance Ratio Cut (RRC)]
		\[
		\mathrm{RRC}(A, B)
		\;=\; \mathrm{ResCut}(A,B)\cdot\left(\frac{1}{|A|} + \frac{1}{|B|}\right).
		\]
	\end{definition}

The resistance ratio cut is a combinatorial optimization problem over discrete vertex partitions. 
As in the classical theory of spectral clustering, such optimization problems are NP-hard in general. 
A standard approach is therefore to replace the discrete indicator vectors by continuous variables, leading to a tractable spectral relaxation. 
The next theorem shows that this relaxation is precisely the variational problem studied in Section \ref{sec:spectral}.    

Note that we are not inventing a new relaxation technology; we are identifying the combinatorial objective whose standard spectral relaxation is the top eigenvector for the resistance Laplacian. 
The main aim is to provide the missing variational counterpart to the Fiedler-like theorem of \cite{gupta}.

\begin{theorem}[Resistance ratio-cut relaxation]
		\label{thm:rrc}
		The discrete optimization problem
		\[
		\max_{\substack{(A,B)\text{ balanced partition of }V}} \mathrm{RRC}(A,B)
		\]
		(where ``balanced'' means $n$ is even and $|A| = |B| = n/2$)
		has as its continuous relaxation the problem
		\[
		\max_{\bx \in \R^n,\; \norm{\bx}=1,\; \bx \perp \1} \bx^\top \calR\,\bx,
		\]
whose maximizers are precisely the dominant eigenvectors of the resistance Laplacian.
	\end{theorem}
	
	\begin{proof}
    
		\textbf{Step 1: Encoding the discrete problem.}
		Given a balanced bipartition $(A,B)$ with $|A|=|B|=n/2$, encode it via the indicator
		vector $\bx \in \R^n$ defined by
		\[
		x_i = \begin{cases} +1 & i \in A,\\ -1 & i \in B. \end{cases}
		\]
		Then $\sum_i x_i = |A| - |B| = 0$, so $\bx \perp \1$, and $\norm{\bx}^2 = n$.
		
		\textbf{Step 2: Expressing RRC via the quadratic form.}
		First, verify that:
		\[
		\bx^\top \calR\,\bx
		= \frac{1}{2}\sum_{i,j} r_{ij}(x_i - x_j)^2.
		\]
		For the $\pm 1$ indicator, $(x_i - x_j)^2 = 4$ when $i$ and $j$ are in opposite
		clusters, and $0$ when in the same cluster.  Since the sum contains both orientations of every
cross-pair and $r_{ij}=r_{ji}$, we obtain
		\[
		\bx^\top \calR \,\bx =
\frac{1}{2}
\left(
\sum_{\substack{i\in A\\j\in B}}
4r_{ij}
+
\sum_{\substack{i\in B\\j\in A}}
4r_{ij}
\right)\\
= 4 \sum_{\substack{i\in A,\,j\in B}} r_{ij}
		= 4\,\mathrm{ResCut}(A,B).
		\]
		Since $|A|=|B|=n/2$:
		\[
		\mathrm{RRC}(A,B) = \mathrm{ResCut}(A,B)\cdot\frac{4}{n} = \frac{1}{n}\,\bx^\top {\calR\,\bx}.
		\]
		Maximizing $\mathrm{RRC}(A,B)$ over balanced bipartitions is thus equivalent to
		\[
		\max_{\bx \in \{+1,-1\}^n,\; \bx\perp\1,\; \norm{\bx}^2 = n}
		\bx^\top \calR\,\bx.
		\]
		
		\textbf{Step 3: Continuous relaxation.}
		Relaxing $x_i \in \{+1,-1\}$ to $x_i \in \R$ and normalizing $\hat\bx = \bx/\norm{\bx}$
		yields the Rayleigh quotient problem:
		\[
		\max_{\hat\bx \in \R^n,\;\norm{\hat\bx}=1,\;\hat\bx\perp\1}
		\hat\bx^\top \calR\,\hat\bx.
		\]
		By the Courant--Fischer theorem, the maximum of the Rayleigh quotient of $\calR$
	over the orthogonal complement of $\ker\calR$ is precisely {$\rho = \lambda_{\max}(\calR)$}, attained at the corresponding eigenvector.	
	\end{proof}
	
	\begin{remark}[Contrast with standard spectral clustering]
The Fiedler vector minimizes $\bx^\top L\bx$ subject to $\norm{\bx}=1$,	$\bx\perp\1$, which relaxes the \emph{minimum} ratio cut.  The top eigenvector of $\calR$ relaxes the \emph{maximum} resistance ratio cut.  
The two objectives are structurally dual: $L$ penalizes placing adjacent vertices in different clusters, while $\calR$ rewards placing globally distant vertices in different clusters.
 By Theorem \ref{thm:vertexenergy}, the spectral relaxation of the resistance ratio cut seeks graph signals that simultaneously maximize 
 \[2w^\top L^{\dagger} w + n \sum_v \delta(v) w(v)^2. \]
Thus, unlike the Fiedler vector, which is determined entirely by the Laplacian geometry, the dominant resistance eigenvector also incorporates the resistance deviation landscape of the graph.
\end{remark}

We give a variational interpretation of the Gupta-Lather-Balaji partition \cite{gupta}. 
Instead of arising only from spectral properties of the resistance Laplacian, this partition is a continuous relaxation of an effective-resistance-based graph cut problem. 
Theorem \ref{thm:variational} shows that its objective splits into two complementary terms: one enforces smoothness with respect to the resistance geometry induced by the ordinary Laplacian, and the other promotes concentrating signal energy on vertices with positive resistance deviations.
Thus, resistance partitioning depends on both global graph connectivity and the distribution of resistance centrality.
Unlike classical spectral clustering, which minimizes edge cuts, the resistance formulation seeks parts that are maximally separated in the graph’s effective-resistance geometry.

	\subsection{Normalized Resistance Cut and the Generalized Eigenproblem}
While the resistance ratio cut balances the cardinalities of the two clusters, in many applications one instead seeks a balance with respect to the resistance geometry itself. This leads naturally to a normalized formulation analogous to the normalized cut of Shi and Malik \cite{ShiMalik2000}.

As in the classical $\mathrm{RatioCut}$ objective, the normalization discourages highly unbalanced partitions.
For notational simplicity, denote the matrix $\mathrm{diag}(R\mathbf{1})$ by $\mathcal{T}$ (its $i$th diagonal entry $t_i$ is the \emph{resistance transmission} of the vertex $i$).  

	\begin{definition}[Normalized resistance cut (NRC)]
For a nonempty set $C\subseteq V$, define its
\emph{transmission volume} by
\[
t(C)=\sum_{i\in C}t_i.
\]
For a nontrivial bipartition
$A\mathbin{\dot\cup}B=V$, the \emph{normalized resistance cut} is
defined by
\[
\mathrm{NRC}(A,B)
=
\mathrm{ResCut}(A,B)
\left(
\frac{1}{t(A)}+\frac{1}{t(B)}
\right).
\]
\end{definition}
Resistance-volume balancing takes into account the aggregate resistance centrality of each cluster. 
Consequently, NRC is better suited to graphs in which resistance centrality varies substantially across the vertex set.    
	\begin{theorem}[Spectral relaxation of NRC]
\label{thm:nrc}
The normalized resistance-cut maximization problem
\[
\max_{\substack{A\sqcup B=V\\A,B\neq\varnothing}}
\mathrm{NRC}(A,B)
\]
has the continuous relaxation
\[
\max_{\substack{\bx\neq\mathbf{0}\\
\bx\perp_{\mathcal{T}}\mathbf{1}}}
\frac{\bx^\top\Rs\bx}{\bx^\top \mathcal{T}\bx},
\]
where
\[
\bx\perp_{\mathcal{T}}\mathbf{1}
\quad\Longleftrightarrow\quad
\mathbf{1}^\top \mathcal{T}\bx=0.
\]
The maximum is the largest generalized eigenvalue
$\mu_{\max}^{(S)}$ of the matrix pair $(\Rs,\mathcal{T})$. Thus, the
maximizers are precisely the generalized eigenvectors satisfying
\[
\Rs\bx=\mu_{\max}^{(\mathcal{T})}\mathcal{T}\bx.
\]
Equivalently,
\[
\mu_{\max}^{(\mathcal{T})}
=
\lambda_{\max}\!\left(\Rs^{\mathrm{sym}}\right),
\qquad
\Rs^{\mathrm{sym}}
=
\mathcal{T}^{-1/2}\Rs \mathcal{T}^{-1/2}.
\]
\end{theorem}

\begin{proof}
Given a nontrivial bipartition
$A\sqcup B=V$, define
\[
x_i=
\begin{cases}
\dfrac{1}{t(A)}, & i\in A,\\[2mm]
-\dfrac{1}{t(B)}, & i\in B.
\end{cases}
\]
Then
\[
\mathbf{1}^\top \mathcal{T}\bx = \sum_{i\in A}\frac{t_i}{t(A)} - \sum_{j\in B}\frac{t_j}{t(B)} = 1-1 = 0,
\]
so $\bx\perp_{\mathcal{T}}\mathbf{1}$.

Using the symmetry of the resistance distances, we obtain
\begin{align*}
\bx^\top\Rs\bx
&= \frac{1}{2} \sum_{i,j=1}^{n}r_{ij}(x_i-x_j)^2\\ 
&= \sum_{\substack{i\in A\\j\in B}} r_{ij}
\left(
\frac{1}{t(A)}+\frac{1}{t(B)}
\right)^2\\
&=
\mathrm{ResCut}(A,B)
\left(
\frac{1}{t(A)}+\frac{1}{t(B)}
\right)^2.
\end{align*}
Moreover,
\[
\bx^\top \mathcal{T}\bx
=
\frac{t(A)}{t(A)^2}
+
\frac{t(B)}{t(B)^2}
=
\frac{1}{t(A)}+\frac{1}{t(B)}.
\]
Therefore,
\[
\frac{\bx^\top\Rs\bx}{\bx^\top \mathcal{T}\bx}
=
\mathrm{ResCut}(A,B)
\left(
\frac{1}{t(A)}+\frac{1}{t(B)}
\right)
=
\mathrm{NRC}(A,B).
\]

Dropping the two-valued constraint on $\bx$ gives the continuous
relaxation
\[
\max_{\substack{\bx\neq\mathbf{0}\\
\bx\perp_{\mathcal{T}}\mathbf{1}}}
\frac{\bx^\top\Rs\bx}{\bx^\top \mathcal{T}\bx}.
\]
Since $r_i>0$ for every vertex, $S$ is positive definite. Under the
substitution
\[
\widehat{\bx}=\mathcal{T}^{1/2}\bx,
\]
we have
\[
\bx^\top \mathcal{T}\bx=\|\widehat{\bx}\|^2
\]
and
\[
\bx^\top\Rs\bx
=
\widehat{\bx}^{\top}
\Rs^{\mathrm{sym}}
\widehat{\bx}.
\]
Furthermore, with
\[
\mathbf{1}_{\mathcal{T}}=\mathcal{T}^{1/2}\mathbf{1},
\]
the orthogonality condition becomes
\[
\mathbf{1}^\top \mathcal{T}\bx=0
\quad\Longleftrightarrow\quad
\widehat{\bx}^{\top}\mathbf{1}_{\mathcal{T}}=0.
\]
Consequently,
\[
\max_{\substack{\bx\neq\mathbf{0}\\
\bx\perp_\mathcal{T}\mathbf{1}}}
\frac{\bx^\top\Rs\bx}{\bx^\top \mathcal{T}\bx}
=
\max_{\substack{\widehat{\bx}\neq\mathbf{0}\\
\widehat{\bx}\perp\mathbf{1}_{\mathcal{T}}}}
\frac{
\widehat{\bx}^{\top}
\Rs^{\mathrm{sym}}
\widehat{\bx}
}{
\|\widehat{\bx}\|^2
}
=
\lambda_{\max}\!\left(\Rs^{\mathrm{sym}}\right)
=
\mu_{\max}^{(\mathcal{T})}.
\]
The maximizers are precisely the eigenvectors associated with
$\lambda_{\max}(\Rs^{\mathrm{sym}})$. Under the transformation
$\widehat{\bx}=\mathcal{T}^{1/2}\bx$, the eigenvalue equation for
$\Rs^{\mathrm{sym}}$ is equivalent to
\[
\Rs\bx=\mu_{\max}^{(\mathcal{T})}\mathcal{T}\bx.
\]
This completes the proof.
\end{proof}

Figure \ref{fig:compareparts} shows the Fiedler and Resistance partitions applied to the $10$-vertex graph introduced earlier. 
The globally remote vertex $8$ and two other nearby vertices which have positive coordinates, form a part and the remaining $6$ vertices form another part, both these parts have roughly equal resistance volume. 
On the other hand, the Fiedler strategy produces parts of equal size. 
\begin{figure}
    \centering
    \includegraphics[width=0.65\linewidth]{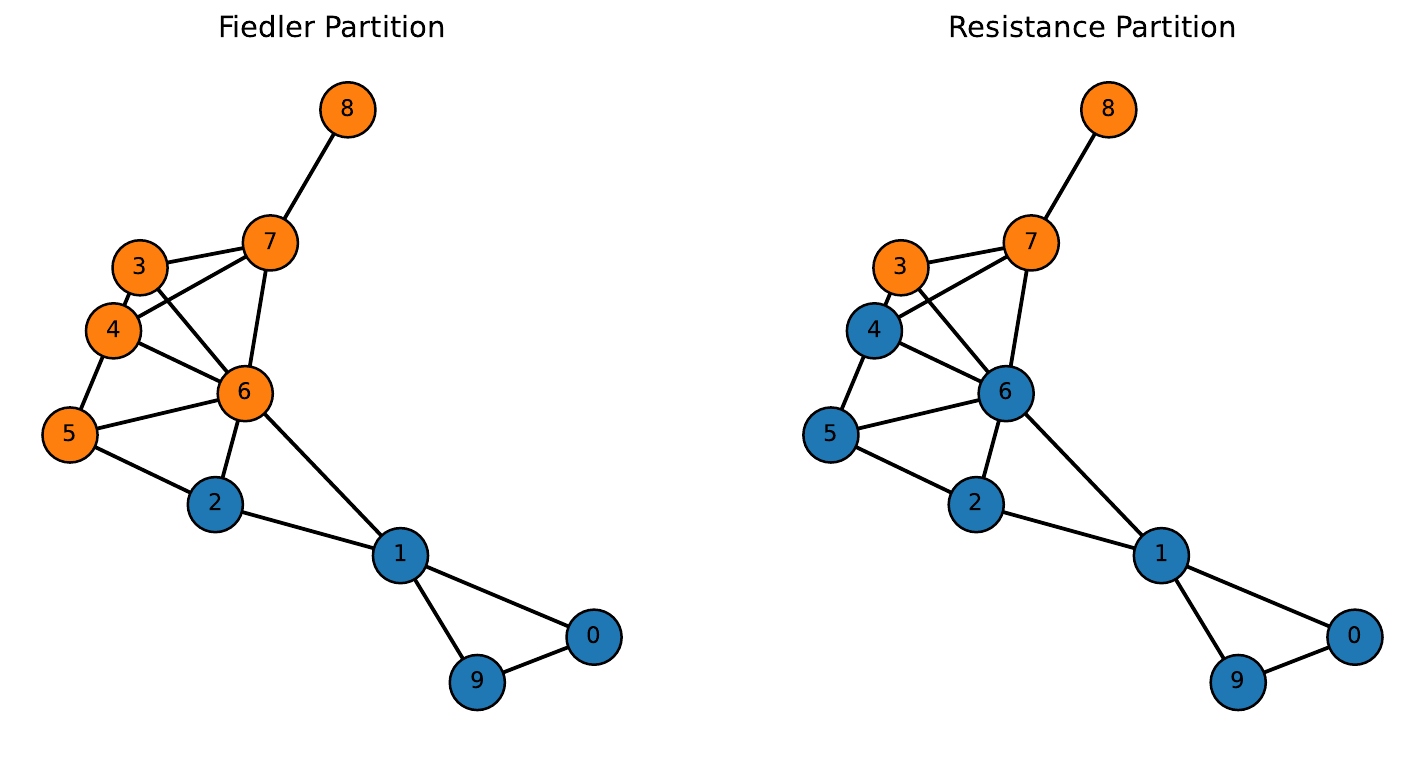}
    \caption{Partition comparison}
    \label{fig:compareparts}
\end{figure}

\subsection{Data clustering}

We use the spectral relaxation from the previous subsection to define a graph partitioning algorithm based on the dominant eigenvector of the resistance Laplacian. Here, we study its behavior on several benchmark datasets and compare it with classical clustering methods, including spectral clustering. Our aim is not just to compare performance, but to understand how the resistance-based embedding reflects the spectral decomposition from Section \ref{sec:spectral}. We focus on when the resistance embedding resembles the classical Fiedler embedding and when resistance deviations lead to a distinctly different partition.

To construct a connected graph for each dataset, we use a weighted $k$-nearest-neighbor graph using the Euclidean distance. 
Edge weights are assigned using a Gaussian similarity kernel,
\[ w_{ij} = \exp\left( -\frac{||x_i - x_j ||^2}{2\sigma^2}\right),\]
where $\sigma$ is chosen according to the local scale of the dataset. 
We then compute both the ordinary graph Laplacian and the resistance Laplacian of the resulting graph. 
The classical spectral clustering algorithm uses the Fiedler vector of the Laplacian, whereas the proposed method uses an eigenvector corresponding to the largest eigenvalue of the resistance Laplacian. 
In both cases, the one-dimensional embedding is partitioned by the sign of the corresponding eigenvector, yielding a bipartition of the dataset.

The theory predicts three clustering regimes according to the size of $M_L$ relative to $D_L$: (i) where $M_L$ is negligible, here resistance and Fiedler eigenvectors are very close; (ii) when $M_L$ is moderate we get same partition but distorted embedding; (iii) when $M_L$ is dominant the energy concentrates on high-$\delta$ vertices. 
The real-life datasets are chosen as empirical instances of these three regimes. 
We now describe these datasets. 

We illustrate the method on several benchmark datasets from the UCI Machine Learning Repository: the \textbf{Iris} dataset of iris flower measurements, the \textbf{Wine} dataset of chemical analyses of wines from different cultivars, and the \textbf{Digits} dataset of handwritten digit images. 
These datasets have varying geometric complexity and thus provide a suitable setting for comparing the resistance embedding with the classical Fiedler embedding.
For each dataset, we compare the partitions from the two spectral methods, visualize the corresponding one-dimensional embeddings, and relate the results to the theoretical predictions of Theorems \ref{thm:variational} and \ref{thm:vertexenergy}.

The Iris dataset contains morphological measurements of iris flowers from three species. 
We consider the binary classification problem involving the Setosa and Versicolor classes, which are known to be almost linearly separable.
This provides a simple yet informative benchmark for comparing the two spectral embeddings.

\begin{figure}
    \centering
    \includegraphics[width=0.72\linewidth]{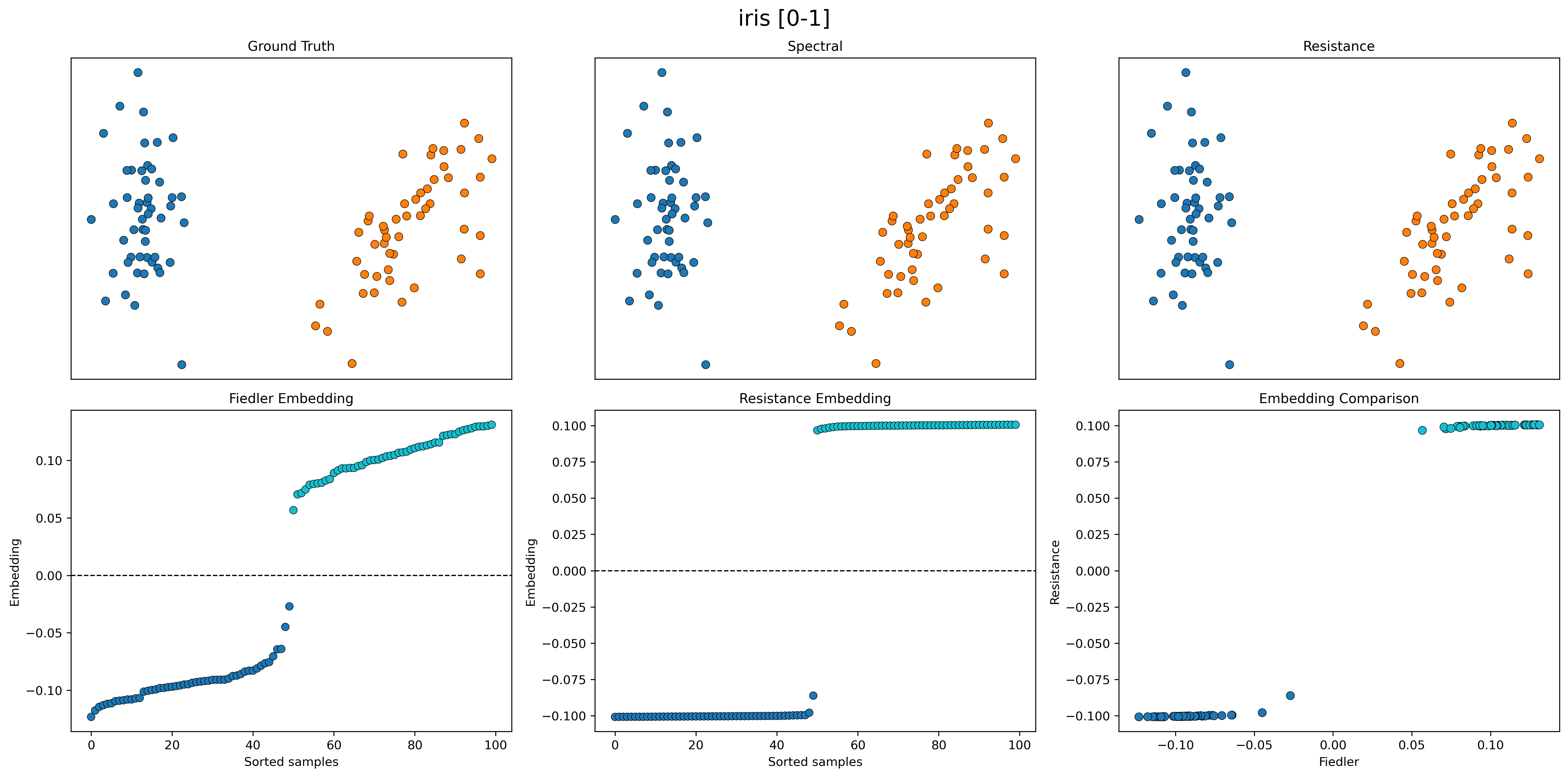}
    \caption{Clustering the Iris dataset}
    \label{fig:iris}
\end{figure}
As shown in Figure \ref{fig:iris}, both the classical spectral method and the resistance-based method almost perfectly recover the true partition. 
Their one-dimensional embeddings, however, differ noticeably: the Fiedler embedding varies smoothly within each cluster, while the resistance embedding is nearly constant within each class, yielding a much sharper separation. 
The embedding comparison plots $(u(i), w(i))$, where $u(i)$ is the $i$th coordinate of the Fiedler vector and $w(i)$ is the $i$th coordinate of the top resistance vector, showing that the two coordinates are strongly but nonlinearly correlated.

This example shows that similar partitions can arise from substantially different spectral embeddings. 
By Theorem \ref{thm:vertexenergy}, resistance optimization balances a smoothness term from the Laplacian pseudoinverse with a heterogeneous term from the resistance deviations. 
Although this heterogeneous correction does not change the partition, it significantly redistributes the embedding coordinates, increasing the contrast between clusters and illustrating how the resistance Laplacian alters the embedding geometry while preserving the partition.

The Wine dataset contains chemical measurements of wines from three cultivars. 
Following the standard benchmark protocol, we focus on the binary classification problem for classes 1 and 2. 
These classes are well separated in feature space, making this dataset a representative example of a graph with simple global geometry.

\begin{figure}
    \centering
    \includegraphics[width=0.75\linewidth]{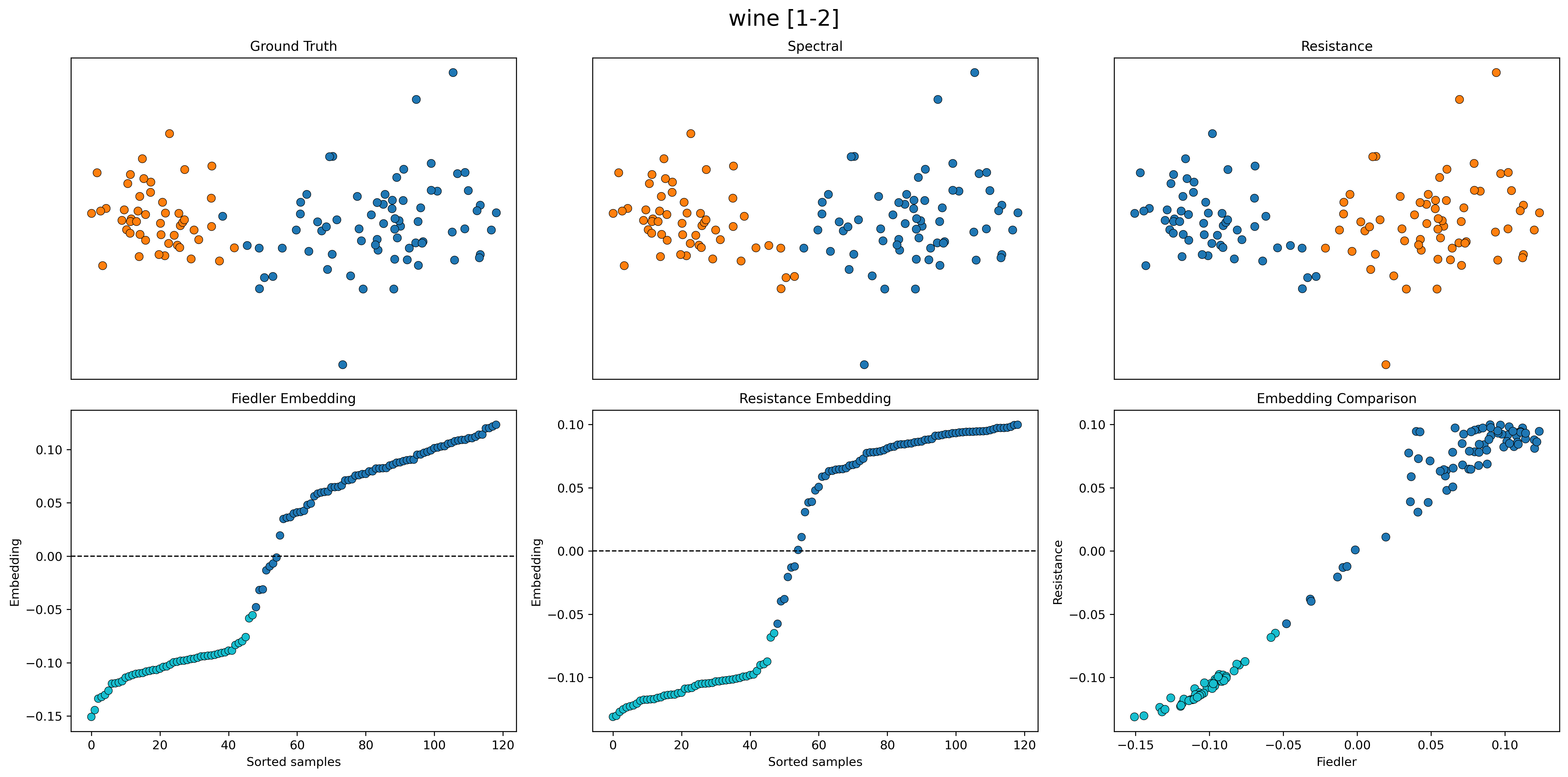}
    \caption{Clustering the Wine dataset}
    \label{fig:wine}
\end{figure}

Figure \ref{fig:wine} compares partitions from the classical Fiedler vector and the dominant resistance eigenvector. 
Both methods recover the class structure and yield nearly identical partitions. 
Their sorted one-dimensional embeddings have similar profiles, with a clear transition between the two clusters, and the embedding comparison plot shows the two coordinates are almost linearly related.

This close agreement implies that resistance deviations introduce only a weak heterogeneous correction for this graph. 
Thus the diagonal term $2D_L$ dominates the optimization in Theorem \ref{thm:eigenbasisoptimization}, keeping the resistance eigenvector close to the classical Fiedler vector. 
Here, the resistance Laplacian essentially preserves the geometry of the ordinary Laplacian while yielding an equivalent partition.

The Digits dataset consists of grayscale images of handwritten digits encoded as pixel-intensity vectors. 
We focus on binary classification of digits $3$ and $8$, whose shapes overlap substantially in feature space, yielding a harder clustering task than before.

\begin{figure}
    \centering
    \includegraphics[width=0.75\linewidth]{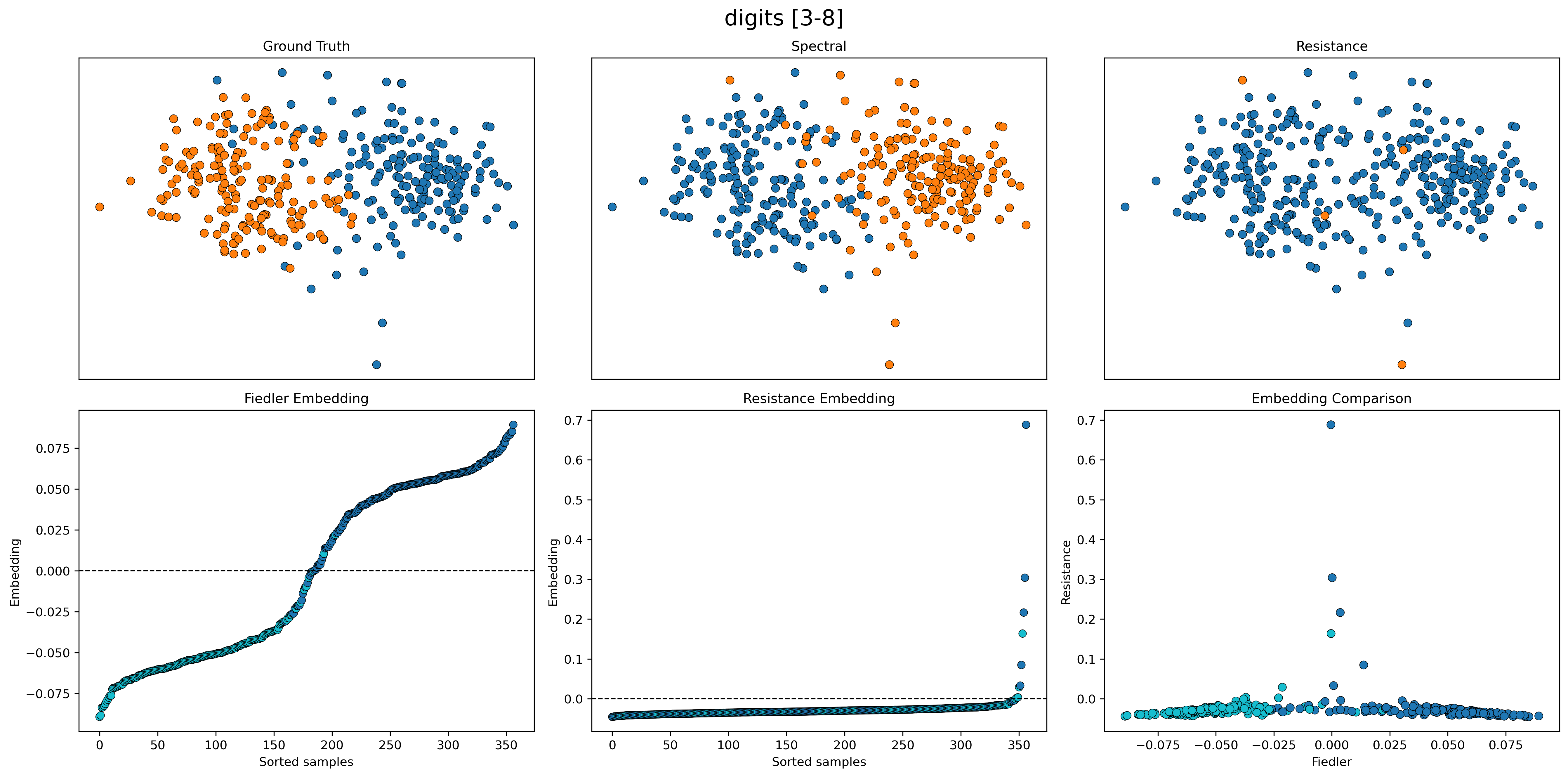}
    \caption{Clustering digits 3 and 8}
    \label{fig:digits}
\end{figure}

Figure \ref{fig:digits} shows a marked contrast between the two spectral methods. 
Classical spectral clustering recovers the two main groups with a moderate number of misclassifications. 
The resistance partition, however, is driven by a few observations with very large embedding values, while most samples cluster near zero. 
This is evident in the sorted resistance embedding and in the comparison plot, where resistance coordinates differ strongly from Fiedler coordinates.

This aligns with the spectral decomposition developed earlier. 
Theorem \ref{thm:vertexenergy} shows that resistance optimization adds a term involving resistance deviations, which favors concentrating signal energy on vertices with large positive deviations. 
Thus the leading resistance eigenvector no longer spreads its energy across the graph but emphasizes a small subset of observations that are globally distant in effective-resistance geometry. 
Instead of approximating the Fiedler embedding, the resistance embedding reflects a different notion of global accessibility, producing a qualitatively different partition.

The three examples above illustrate the regimes predicted by the spectral theory in Section \ref{sec:spectral}. 
When resistance deviations are nearly uniform (Wine), the resistance and Fiedler embeddings are almost identical. 
As the deviation term becomes more influential (Iris), the embedding changes but the partition is preserved. 
When heterogeneity dominates (Digits), the resistance embedding diverges sharply from the Fiedler embedding and highlights vertices that are globally peripheral in effective-resistance geometry. 
These results support the earlier decomposition and show that the resistance Laplacian yields a genuinely new geometric representation of graph-structured data, not just another clustering method.

\subsection{An exploratory anomaly detection experiment}
Anomaly detection is a fundamental problem in data analysis, with a broad range of approaches based on statistical, distance, density, and structural notions of abnormality; see, for example, the survey of Chandola, Banerjee, and Kumar \cite{Chandola2009}.
When the observations are represented by a graph, the graph structure itself provides additional information that can be exploited for detecting anomalous observations; see Akoglu, Tong, and Koutra \cite{Akoglu2015} for a survey of graph-based anomaly detection methods. 
More generally, graph-based anomaly detection exploits structural information in the graph to identify observations that are atypical relative to the rest of the data.
From this perspective when we consider the interpretation of the resistance Laplacian developed earlier, it is natural to ask whether its dominant eigenvector can provide a spectral signature of globally isolated observations.

The above question is motivated by existing spectral approaches to outlier detection, in which data are represented by a similarity graph and anomalous observations are identified through the spectral structure of the associated graph Laplacian; see, for example, \cite{dang2013} and \cite{yu2010}. 
In these approaches, weakly connected or isolated observations can manifest themselves through distinctive entries or supports of Laplacian eigenvectors. 
Our perspective differs in that we replace the usual local similarity-based notion of connectivity by effective resistance, and hence investigate whether a resistance-Laplacian eigenvector can reveal anomalies that are globally poorly accessible within the data graph.
Our goal is not to propose a new anomaly detection algorithm, but to provide preliminary evidence that the proposed spectral embedding captures global isolation and may be useful for anomaly detection.

We test this on two synthetic datasets: Gaussian blobs and the classical two-moons data. 
In each, we inject three artificial outliers, build a weighted $k$-nearest neighbor graph, form the resistance Laplacian, and compute its dominant eigenvector $w$. 
Empirically, isolated vertices tend to have small coordinates in this eigenvector, so we assign each vertex the score
\[s(v)= w_{\max} - |w(v)|,\]
where $w_{\max}$ is the maximum eigenvector coordinate. 
Vertices with larger $s(v)$ are deemed more anomalous. 
This score is only a simple quantitative summary of the experiments, not a thorough comparison with existing anomaly detection methods.

\begin{figure}
    \centering
    \includegraphics[width=0.72\linewidth]{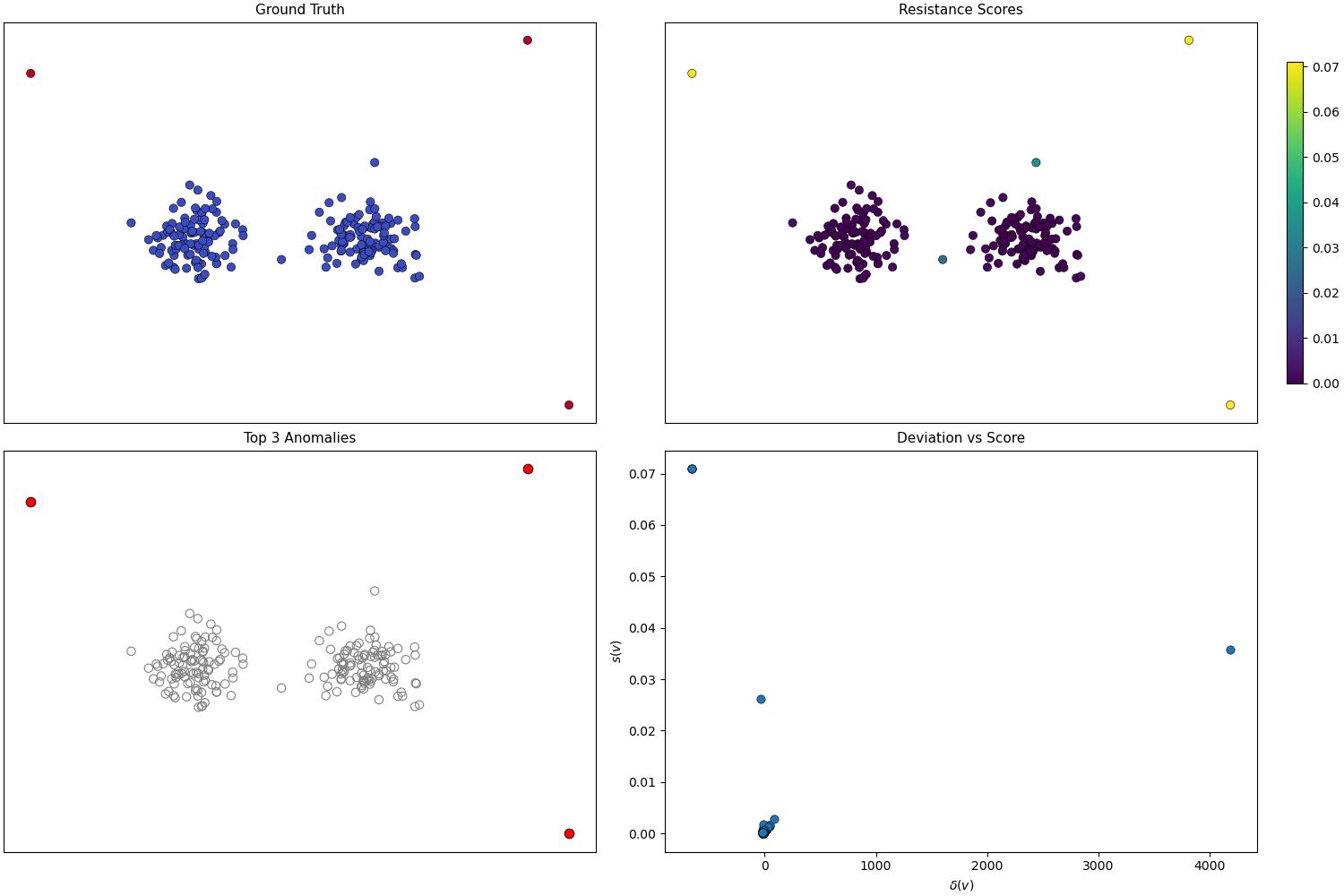}
    \caption{Anomaly detection in two blobs dataset}
    \label{fig:blobs}
\end{figure}

Figure \ref{fig:blobs} shows two Gaussian clusters with three added outliers. The resistance embedding clearly separates the outliers from the dense clusters and assigns them the highest anomaly scores, so all three injected outliers appear among the top three ranked observations.

The scatter plot of resistance deviations versus the dominant resistance coordinates offers further insight. The outliers lie at extreme positions and have much smaller eigenvector coordinates than typical vertices. This indicates a strong link between electrical isolation and the dominant resistance eigenvector, though the exact theoretical basis is still unclear.

\begin{figure}
    \centering
    \includegraphics[width=0.7\linewidth]{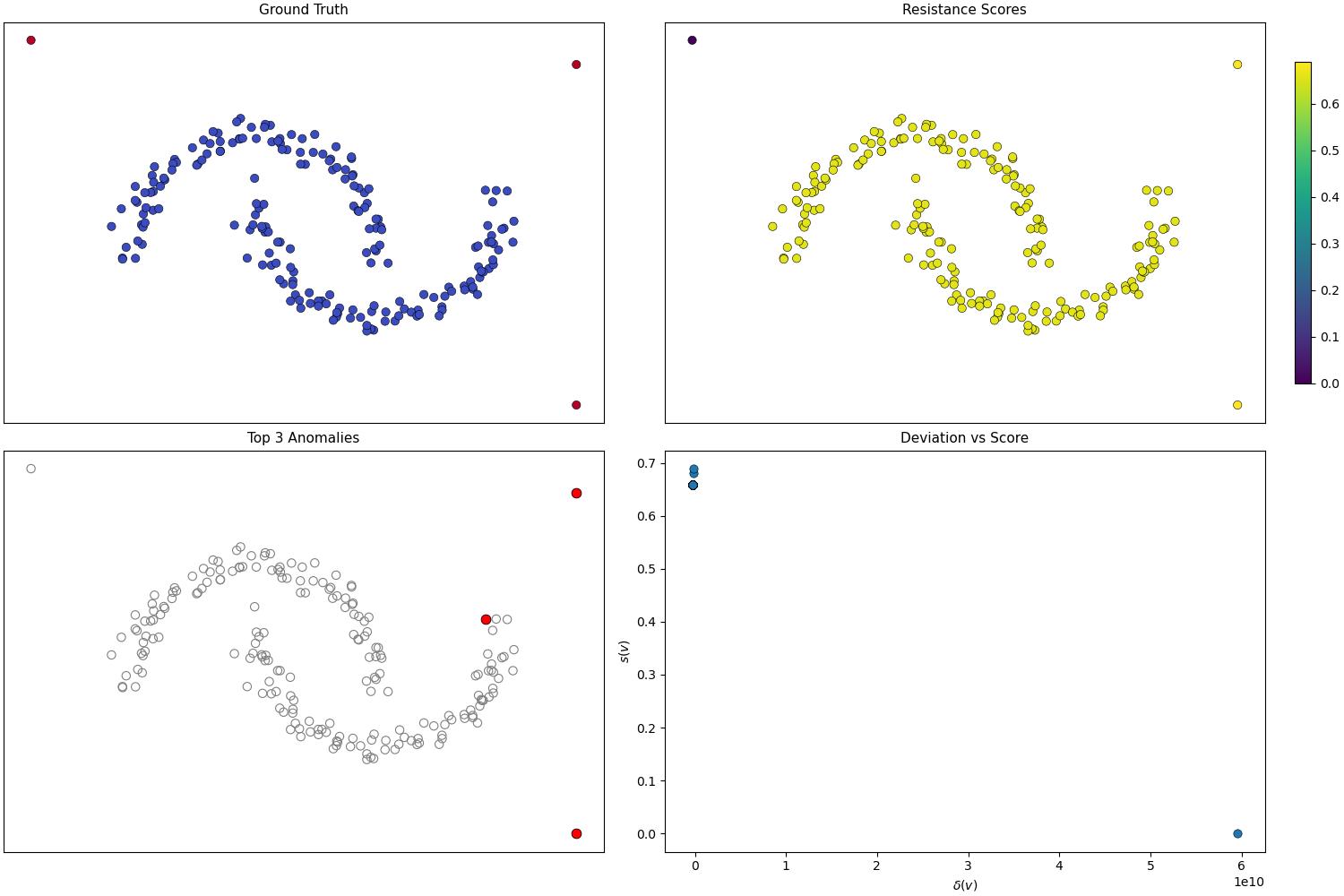}
    \caption{Anomaly detection in the two moons dataset}
    \label{fig:moons}
\end{figure}

Figure \ref{fig:moons} repeats the experiment on the nonlinear two-moons dataset, whose geometry is substantially more challenging than in the Gaussian example. 
The resistance embedding correctly identifies two added outliers, while the third is ranked similarly to naturally occurring peripheral points, indicating that the proposed score is sensitive to both Euclidean distance and intrinsic graph structure.
In particular, the deviation score of the unidentified outlier is near-zero, hence it doesn't receive an anomolous score. 

These preliminary results are encouraging but show that considerably more work is needed before the spectral score can be considered competitive for anomaly detection. Open problems include improved graph construction, alternative score functions, normalization schemes, and theoretical guarantees linking the dominant resistance eigenvector to graph-theoretic isolation.

These experiments should thus be viewed as proof-of-concept rather than a comprehensive empirical study. Their main purpose is to show that the dominant resistance eigenvector captures meaningful information about globally isolated vertices and to motivate further study of resistance-Laplacian-based anomaly detection methods.

\section{Concluding remarks}\label{sec:conclusion}
In this paper, we introduced the resistance Laplacian as a spectral
object encoding the global resistance geometry of a connected graph.
We analyzed its structural and spectral properties, determined the
spectrum and inertia of the resistance-deviation operator, and
interpreted the dominant eigenspace of the resistance Laplacian through
the associated resistance embedding. These results show that the resistance Laplacian offers a geometric viewpoint complementary to combinatorial and normalized Laplacians, emphasizing global accessibility rather than local adjacency.

Preliminary clustering and anomaly detection experiments suggest that the dominant resistance embedding captures meaningful geometric structure in graph-structured data. Although exploratory, these results indicate that resistance-based spectral methods have promise for data analysis and may yield useful algorithmic tools.

Computationally, a key limitation is that the resistance Laplacian is typically dense, making direct eigen-decomposition expensive for large graphs. Our algorithms are therefore aimed at moderate-sized graphs and at illustrating the theory. Developing scalable methods that avoid explicitly forming and diagonalizing dense resistance matrices is an important direction for future work, for example via iterative eigen-solvers using matrix-vector products, low-rank approximations, or graph sparsification. Such advances would extend resistance-based spectral methods to large-scale graph learning.

This work also raises several theoretical and algorithmic questions for further study.

\begin{enumerate}
    \item \textbf{Spectral characterization of resistance embeddings}: The experiments in this paper indicate a close link between resistance deviations and the coordinates of the dominant resistance eigenvector. Formally characterizing this relationship and developing graph-theoretic interpretations of resistance embeddings remains an open theoretical problem, one that could clarify the geometric information encoded by the resistance Laplacian.
    \item \textbf{Resistance-based anomaly scores}: Preliminary anomaly detection experiments suggest that the dominant resistance embedding captures information about globally isolated vertices. However, the anomaly score used here is heuristic and only exploratory. Developing principled scoring functions based on the resistance spectrum, with theoretical guarantees and stability results, is an important direction for future work.
    \item \textbf{Scalable numerical algorithms}: The dense nature of the resistance Laplacian motivates the development of efficient computational techniques capable of handling graphs with millions of vertices. Matrix-free iterative eigensolvers, randomized approximation methods, hierarchical representations, and distributed algorithms offer promising avenues for reducing both computational complexity and memory requirements while preserving the essential spectral information.
    \item \textbf{Applications in graph learning and data analysis}: The resistance Laplacian offers a natural global accessibility measure that complements existing spectral methods. A systematic empirical comparison of resistance-based embeddings and classical graph Laplacians on tasks such as clustering, anomaly detection, graph partitioning, manifold learning, and graph representation learning would clarify the strengths and limitations of this new spectral framework.
\end{enumerate}

\section*{Acknowledgments}
Both the authors are partially supported by a grant from the Infosys Foundation. 
The authors would like to thank Shreya Singh, who developed some of the clustering experiments as part of her course project in January -- April 2026 semester. 
The Python codes were further improved with the help of  ChatGPT Go. 
AI tools from OverLeaf were used to improve readability and grammatical flow.

\bibliographystyle{amsplain}
\bibliography{refs}

\end{document}